\documentclass[10pt]{article}%
\usepackage{amsmath}
\usepackage{amsfonts}
\usepackage{mathrsfs}
\usepackage{amssymb, color}
\usepackage[linkcolor=black,anchorcolor=black,citecolor=black]{hyperref}
\usepackage{graphicx}
\numberwithin{equation}{section}
\usepackage[body={15.5cm,21cm}, top=3cm]{geometry}%
\providecommand{\U}[1]{\protect\rule{.1in}{.1in}}
\providecommand{\U}[1]{\protect \rule{.1in}{.1in}}
\newtheorem{theorem}{Theorem}[section]

\newtheorem{corollary}[theorem]{Corollary}

\newtheorem{definition}[theorem]{Definition}
\newtheorem{example}[theorem]{Example}

\newtheorem{lemma}[theorem]{Lemma}

\newtheorem{proposition}[theorem]{Proposition}
\newtheorem{remark}[theorem]{Remark}

\newenvironment{proof}[1][Proof]{\noindent \textbf{#1.} }{\  \rule{0.5em}{0.5em}}
\DeclareMathOperator*{\esssup}{ess\,sup}
\DeclareMathOperator*{\essinf}{ess\,inf}
\begin{document}
	
	\title{Optimal Multiple  Stopping Problem under Nonlinear Expectation}
	\author{ Hanwu Li\thanks{Center for Mathematical Economics, Bielefeld University,		hanwu.li@uni-bielefeld.de. This research was supported by the German Research Foundation (DFG) via CRC 1283.}}
	\date{}
	\maketitle
	
	\begin{abstract}
		In this paper, we study the optimal multiple stopping problem under the filtration consistent nonlinear expectations. The reward is given by a set of random variables satisfying some appropriate assumptions rather than an RCLL process. We first construct the optimal stopping time for the single stopping problem, which is no longer given by the first hitting time of processes. We then prove by induction that the value function of the multiple stopping  problem can be interpreted as the one for the single stopping problem associated with a new reward family, which allows us to construct the optimal multiple stopping times. If the reward family satisfies some strong regularity conditions, we show that the reward family and the value functions can be aggregated by some progressive processes. Hence, the optimal stopping times can be represented as hitting times.
	\end{abstract}
	
	\textbf{Key words}: nonlinear expectations, optimal stopping, multiple optimal stopping, aggregation. 
	
	\textbf{MSC-classification}: 60G40
	
	\section{Introduction}
	The optimal single stopping problem, both under uncertainty and ambiguity (or Knightian uncertainty, especially drift uncertainty), has attracted a great deal of attention and been well studied. We may refer to the papers \cite{BY1}, \cite{BY2}, \cite{CR}, \cite{PS}. Consider a filtered probability space $(\Omega,\mathcal{F},\{\mathcal{F}_t\}_{t\in[0,T]},P)$ satisfying the usual conditions of right-continuity and completeness. Given a nonnegative and adapted reward process $\{X_t\}_{t\in[0,T]}$ with some integrability and regularity conditions, we then define
	\begin{displaymath}
		V_0=\sup_{\tau\in\mathcal{S}_0}\mathcal{E}[X_\tau],
	\end{displaymath}
	where $\mathcal{S}_0$ is the collection of all stopping times taking values between $0$ and $T$. The operator  $\mathcal{E}[\cdot]$ corresponds to the classical expectation $E[\cdot]$ when the agent only faces risk or uncertainty (i.e., he does not know the future state but exactly knows the distribution of the reward process) while it corresponds to some nonlinear expectation if ambiguity is taken into account (i.e., the agent even has no full confidence about the distribution). Under both situations, the main objective is to compute the value $V_0$ as explicitly as possible and find some stopping time $\tau^*$ at which the supremum is attained, that is, $V_0=\mathcal{E}[X_{\tau^*}]$. For this purpose, consider the value function
	\begin{displaymath}
		V_t=\esssup_{\tau\in\mathcal{S}_t}\mathcal{E}_t[X_\tau],
	\end{displaymath}
	where $\mathcal{S}_t$ is the set of stopping times greater than $t$. When assuming some regularity of the reward family $\{X_t\}_{t\in[0,T]}$ and some appropiate conditions of the nonlinear conditional expectation $\mathcal{E}_t[\cdot]$, we prove that the process $\{V_t\}_{t\in[0,T]}$ admits an RCLL modification, for simplicity, still denoted by $\{V_t\}_{t\in[0,T]}$. Furthermore, the stopping time given in terms of the first hitting time
	\begin{displaymath}
		\tau=\inf\{t\geq 0: V_t=X_t\}
	\end{displaymath}
	is optimal and $\{V_t\}_{t\in[0,T]}$ is the smallest $\mathcal{E}$-supermartingale (which reduces to the classical supermartingale when $\mathcal{E}[\cdot]$ is the linear expectation) dominating the reward process $\{X_t\}_{t\in[0,T]}$. One of the most important applications of the single optimal stopping problem is pricing for American options.
	
	Motivated by the pricing for financial derivatives with several exercise rights in the energy market (swing options), one needs to solve an optimal multiple stopping problem. Mathematically, given a reward process $\{X_t\}_{t\in[0,T]}$, if an agent has $d$ exercise rights, the price of this contract is defined as follows:
	\begin{displaymath}
		v_0=\sup_{(\tau_1,\cdots,\tau_d)\in\mathcal{\widetilde{S}}_0^d}E[\sum_{i=1}^{d}X_{\tau_i}].
	\end{displaymath}
	To avoid triviality, we assume that there exists a constant $\delta>0$, which represents the length of the reftracting time interval, such that the difference of any two successive exercises is greater than $\delta$. Therefore, $\widetilde{S}_0^d$ is the collection of stopping times $(\tau_1,\cdots,\tau_d)$ such that $\tau_1\geq 0$ and $\tau_j-\tau_{j-1}\geq \delta$, for any $j=2,\cdots,d$. There are several papers concerning this kind of problem. To name a few, \cite{BS}, \cite{MH} mainly deal with the discrete time case focusing on the Monto Carlo methods and algorithm, \cite{CT} investigates the continuous time case and the time horizon can be both finite and infinite. It is worth pointing out that all the existing literature does not consider the multiple stopping problem under Knightian uncertainty.

	In fact, to make the value function well-defined for both the single and multiple stopping problems, the reward can be given by a set of random variables $\{X(\tau),\tau\in\mathcal{S}_0\}$ satisfying some compatibility properties, which means that we do not need to assume that the reward family can be aggregated into a progressive process. Under this weaker assumption on reward family, \cite{E} and \cite{KQR} established the existence of the optimal stopping times for the single stopping problem and multiple stopping problem respectively. Without aggregation of the reward family and the value function, the optimal stopping time is no longer given by the first hitting time of processes but by the essential infimum over an appropriate set of stopping times.
	
	In the present work, we study the multiple stopping problem under Knightian uncertainty without the requirement of aggregation of the reward family. We will use the filtration consistent nonlinear expectations established in \cite{BY1} to model Knightian uncertainty. First, we focus on the single stopping problem.  Similar with the classical case, the value function is a kind of nonlinear supermartingale which is the smallest one dominating the reward family. Besides, the value function shares the same regularity with the reward family in the single stopping case. Applying an approximation method, we prove the existence of the optimal stopping times under the assumption that the reward family is continuous along stopping times under nonlinear expectation (see Definition \ref{d}). It is important to note that in proving the existence of optimal stopping times, we need the assumption that the nonlinear expectation is sub-additive and positive homogenous, which is to say that the nonliear expectation is an upper expectation. Hence, this optimal stopping problem is in fact a ``$\sup_\tau \sup_P$" problem.
	
	For the multiple stopping case, one important observation is that the value function of the $d$-stopping problem coincides with the one of the single stopping case corresponding to a new reward family, where the new reward family is given by the maximum of a set of value functions associated with the $(d-1)$-stopping problem. Therefore, we may construct the optimal stopping times by an induction mehtod providing that this new reward family satisfies the conditions under which the optimal single stopping time exists.  This is the main difficulty lies in this problem due to some measurability issues. To overcome this problem, we need to slightly modify the reward family to a new one and to establish the regularity of the induced value functions.
	
	Recall that, in \cite{BY2} and \cite{CR} for the single stopping problems under Knightian uncertainty, the reward is given by an RCLL, adapted process and the optimal stopping time can be represented as first hitting time, which provides an efficient way to calculation an optimal stopping time. In our setting, if the reward family satisfies some stronger regularity conditions than those made in the existence result, we may prove that the reward family and the associated value function can be aggregated into some progressively measurable processes. Therefore, in this case, the optimal stopping times can be interpreted in terms of hitting times of processes.
	
	The paper is organized as follows. In Section 2, we  investigate the properties of the value function and construct the optimal stopping times for  the optimal single stopping problem under nonlinear expectations. Then we solve the  optimal double stopping problem under nonlinear expectations in Section 3.  In Section 4, we generalize this result to the optimal $d$-stopping problem. When the reward family satisfies some strong regularity, we study some aggregation results in Secion 5 and then interpret the optimal stopping times as the first hitting times of processes.
	
\section{The optimal single stopping problem under nonlinear expectation}	
	In this paper, we fix a finite time horizon $T>0$. Let $(\Omega,\mathcal{F},P)$ be a complete probability space equipped with a filtration $\mathbb{F}=\{\mathcal{F}_t\}_{t\in[0,T]}$ satisfying the usual conditions of right-continuity and completeness. We deonte by $L^0(\mathcal{F}_T)$ the collection of all $\mathcal{F}_T$ measurable random variables. An $\mathbb{F}$-expectation is a pair $(\mathcal{E},\textrm{Dom}(\mathcal{E}))$, where $\mathcal{E}$ is a nonlinear operator defined on its domain Dom$(\mathcal{E})$ (for the definition, we may refer to Definition \ref{D2.1} and \ref{D2.2}). Throughout this paper, we assume that the $\mathbb{F}$-expectation satisfies the hypotheses (H0)-(H4) in the Appendix and the following condition:
	\begin{description}
		\item[(H5)] if the sequence $\{\xi_n\}_{n\in\mathbb{N}}\subset \textrm{Dom}^+(\mathcal{E})$ converges to $\xi\in L^0(\mathcal{F}_T)$, a.s. and satisfies $\liminf_{n\rightarrow\infty}\mathcal{E}[\xi_n]<\infty$, then we have $\xi\in\textrm{Dom}^+(\mathcal{E})$.
	\end{description}
    This assumption is mainly used to prove the following lemma.	

    \begin{lemma}\label{L4.1}
    	Let $\Xi$ be a subset of $\textrm{Dom}^+(\mathcal{E})$.	Suppose that $\sup_{\xi\in\Xi}\mathcal{E}[\xi]<\infty$. Set $\eta=\esssup_{\xi\in\Xi}\xi$. Then, we have $\eta\in\textrm{Dom}^+(\mathcal{E})$.
    \end{lemma}

    \begin{proof}
    	By the definition of essential supremum, there exists a sequence $\{\xi_n\}_{n\in\mathbb{N}}\subset \Xi$ such that $\xi_n\rightarrow \eta$, a.s. Since $\liminf_{n\rightarrow\infty}\mathcal{E}[\xi_n]\leq \sup_{\xi\in\Xi}\mathcal{E}[\xi]<\infty$, Assumption (H5) implies that $\eta\in \textrm{Dom}^+(\mathcal{E})$. 
    \end{proof}

	It is worth pointing out that the $\mathbb{F}$-expectation satisfying (H0)-(H4) preserves almost all properties as the classical expectation, such as strict monotonicity, translation invariance, time consistency, local property. For more details, we may refer to the Appendix.
	
	\begin{example}\label{e1}
		The following pairs are $\mathbb{F}$-expectations satisfying (H0)-(H4):
		\begin{description}
			\item[(1)] $(\{E_t[\cdot]\}_{t\in[0,T]}, L^1(\mathcal{F}_T))$: the classical expectation $E$;
			\item[(2)] $(\{\mathcal{E}^g_t[\cdot],L^2(\mathcal{F}_T)\})$: the $g$-expectation with Lipschitz generator $g(t,z)$ which is progressively measurable, square integrable and satisfies $g(t,0)=0$ (see \cite{BY2},  \cite{CHMP});
			\item[(3)] $(\{\mathcal{E}^g_t[\cdot],L^e(\mathcal{F}_T)\})$: the $g$-expectation with convex generator $g(t,z)$ having quadratic growth in $z$ and satisfying $g(t,0)=0$, where $L^e(\mathcal{F}_T):=\{\xi\in L^0(\mathcal{F}_T): E[\exp(\lambda|\xi|)]<\infty, \forall \lambda>0\}$ (see \cite{BY2}).
		\end{description}
	\end{example}

    \begin{remark}
    	The classical expectation naturally satisfies Assumption (H5) by using the Fatou lemma. However, for the $g$-expetation, this condition may not hold. We refer to Example 5.1 in \cite{BY2} as a counterexample.
    \end{remark}
	
	Throughout this paper, for each fixed stopping time $\tau$, $\mathcal{S}_\tau$ represents the collection of all stopping times taking values between $\tau$ and $T$. We now introduce the definition of admissible family which can be interpreted as the payoff process in the classical case.
	\begin{definition}
		A family of random variables $\{X(\tau),\tau\in\mathcal{S}_0\}$ is said to be admissible if the following conditions are satisfied:
		\begin{description}
			\item[(1)] for all $\tau\in\mathcal{S}_0$, $X(\tau)\in\textrm{Dom}^+_\tau(\mathcal{E}) $;
			\item[(2)] for all $\tau,\sigma\in\mathcal{S}_0$, we have $X(\tau)=X(\sigma)$ a.s. on the set $\{\tau=\sigma\}$.
		\end{description}
	\end{definition}

    \begin{remark}
    	Since the $\mathbb{F}$-expectation is translation invariance, all the results in this paper still hold if the family of random viables $\{X(\tau),\tau\in\mathcal{S}_0\}$ is bounded from below.
    \end{remark}
	
	Now consider the reward given by the admissible family $\{X(\tau),\tau\in\mathcal{S}_0\}$. For each $S\in\mathcal{S}_0$, the value function at time $S$ takes the following form:
	\begin{equation}\label{1.1}
	v(S)=\esssup_{\tau\in\mathcal{S}_S}\mathcal{E}_S[X(\tau)].
	\end{equation}
	
	\begin{definition}
		For each fixed $S\in\mathcal{S}_0$, an admissible family $\{X(\tau),\tau\in\mathcal{S}_S\}$ is said to be an $\mathcal{E}$-supermartingale system (resp. an $\mathcal{E}$-martingale system) if, for any $\tau,\sigma\in\mathcal{S}_S$ with $\tau\leq \sigma$ a.s., we have
		\begin{displaymath}
		\mathcal{E}_\tau[X(\sigma)]\leq X(\tau), \textrm{ a.s. } (\textrm{resp., } 	\mathcal{E}_\tau[X(\sigma)]= X(\tau), \textrm{ a.s. }).
		\end{displaymath}
	\end{definition}
	
	\begin{proposition}\label{P1}
		If $\{X(\tau),\tau\in\mathcal{S}_0\}$ is an admissible family with $\sup_{\tau\in\mathcal{S}_0}\mathcal{E}[X(\tau)]<\infty$, then the value function $\{v(S), S\in\mathcal{S}_0\}$ defined by \eqref{1.1} satisfies the following properties: 
		\begin{description}
			\item[(i)] $\{v(S), S\in\mathcal{S}_0\}$ is an admissible family;
			\item[(ii)] $\{v(S), S\in\mathcal{S}_0\}$ is the smallest $\mathcal{E}$-supermartingale system which is greater than $\{X(S), S\in\mathcal{S}_0\}$;
			\item[(iii)] for any $S\in\mathcal{S}_0$, we have \begin{equation}\label{1.3}
			\mathcal{E}[v(S)]=\sup_{\tau\in\mathcal{S}_S}\mathcal{E}[X(\tau)].
			\end{equation}
		\end{description}
	\end{proposition}
	
	\begin{proof}
	(i)	By Lemma \ref{L4.1}, we have $v(S)\in \textrm{Dom}^+(\mathcal{E})$. It is obvious that $\mathcal{E}_S[X(\tau)]$ is  $\mathcal{F}_S$-measurable for any $S\in\mathcal{S}_0$ and $\tau\in\mathcal{S}_S$. So is $v(S)$. Therefore, we obtain that $v(S)\in \textrm{Dom}^+_S(\mathcal{E})$.
		
		 Now consider two stopping times $\tau,\sigma\in\mathcal{S}_0$. For each $\theta\in\mathcal{S}_\tau$, set $\theta_A=\theta I_A+TI_{A^c}$, where $A=\{\tau=\sigma\}\in \mathcal{F}_\tau\cap \mathcal{F}_\sigma$. Noting that $\theta_A\in\mathcal{S}_\sigma$, it is easy to check that
		\begin{displaymath}
		\mathcal{E}_\tau[X(\theta)]I_A=\mathcal{E}_\tau[X(\theta)I_A]=\mathcal{E}_\tau[X(\theta_A)I_A]=\mathcal{E}_\tau[X(\theta_A)]I_A=\mathcal{E}_\sigma[X(\theta_A)] I_A\leq v(\sigma)I_A.
		\end{displaymath}
		Taking essential supremum over all $\theta\in\mathcal{S}_\tau$ yields that $v(\tau)\leq v(\sigma)$ on the set $A$. By a similar analysis, we have $v(\tau)\geq v(\sigma)$ on the set $A$. Therefore, $\{v(S), S\in\mathcal{S}_0\}$ is an admissible family.
		
	(ii)	We first claim that the set $\{\mathcal{E}_S[X(\tau)],\tau\in\mathcal{S}_S\}$ is upward directed for any $S\in\mathcal{S}_0$. Indeed, for any $\tau,\sigma\in\mathcal{S}_S$, set $\theta_B=\tau I_B+\sigma I_{B^c}$, where $B=\{\mathcal{E}_S[X(\tau)]\geq \mathcal{E}_S[X(\sigma)]\}\in\mathcal{F}_S$. It is easy to check that $\theta_B\in\mathcal{S}_S$ and
	\begin{displaymath}
	\mathcal{E}_S[X(\theta_B)]=\mathcal{E}_S[X(\tau)]I_B+\mathcal{E}_S[X(\sigma)] I_{B^c}=\mathcal{E}_S[X(\tau)]\vee \mathcal{E}_S[X(\sigma)].
	\end{displaymath}
	Hence, the claim holds true. It follows that for any $S\in\mathcal{S}_0$, there exists a sequence of stopping times $\{\tau^n\}_{n\in\mathbb{N}}\subset \mathcal{S}_S$, such that $\mathcal{E}_S[X(\tau^n)]$ converges monotonically up to $v(S)$. For stopping times $S,\sigma\in\mathcal{S}_0$ with $S\geq \sigma$, by Fatou's Lemma (see Proposition \ref{P2.8}), we have
		\begin{equation}\label{1.0}
		\mathcal{E}_\sigma[v(S)]\leq \liminf_{n\rightarrow\infty}\mathcal{E}_\sigma[\mathcal{E}_S[X(\tau^n)]]=\liminf_{n\rightarrow\infty}\mathcal{E}_\sigma[X(\tau^n)]\leq v(\sigma),
		\end{equation}
		which implies that $\{v(S), S\in\mathcal{S}_0\}$ is an $\mathcal{E}$-supermartingale system.
		
		Now we show that $\{v(S), S\in\mathcal{S}_0\}$ is the smallest $\mathcal{E}$-supermartingale system dominating the family $\{X(S),S\in\mathcal{S}_0\}$. It is trivial to check that $v(S)\geq X(S)$ for any $S\in\mathcal{S}_0$. Suppose that $\{v'(S),S\in\mathcal{S}_0\}$ is another $\mathcal{E}$-supermartingale system which is greater than $\{X(S),S\in\mathcal{S}_0\}$. Given $S\in\mathcal{S}_0$, for each $\tau\in\mathcal{S}_S$, we  obtain that
		\begin{displaymath}
		v'(S)\geq \mathcal{E}_S[v'(\tau)]\geq \mathcal{E}_S[X(\tau)].
		\end{displaymath}
		Taking essential supremum over all $\tau\in\mathcal{S}_S$ yields that $v'(S)\geq v(S)$. 
		
		(iii) By choosing $\sigma=0$ in Equation \eqref{1.0}, we have
		\begin{displaymath}
		\mathcal{E}[v(S)]\leq \liminf_{n\rightarrow\infty}\mathcal{E}[X(\tau^n)]\leq \sup_{\tau\in\mathcal{S}_S}\mathcal{E}[X(\tau)].
		\end{displaymath}
		The inverse inequality holds true due to the fact that $v(S)\geq \mathcal{E}_S[X(\tau)]$ for any $\tau\in\mathcal{S}_S$. The proof is complete. 
	\end{proof}
	
	\begin{remark}\label{R1}
		(1) Compared with the value function defined in \cite{BY2} and \cite{CR}, there may not exist an adapted process $\{v_t,0\leq t\leq T\}$ which aggregates the admissible family $\{v(\tau),\tau\in\mathcal{S}_0\}$ in the sense that $v_\tau=v(\tau)$ a.s. for any $\tau\in\mathcal{S}_0$.
		
		(2) It follows from Equation \eqref{1.3} that 
		\begin{displaymath}
			\sup_{S\in\mathcal{S}_0}\mathcal{E}[v(S)]\leq \sup_{\tau\in\mathcal{S}_0}\mathcal{E}[X(\tau)]<\infty.
		\end{displaymath}
		Consequently, we obtain that $v(S)<\infty$, a.s. for any $S\in\mathcal{S}_0$.
		
		(3) The Assumption (H5) is mainly used to make sure that the value function $v(S)$ at any stopping time $S$ belongs to $\textrm{Dom}^+(\mathcal{E})$. We can drop this assumption by requiring that the admissible family satisfies $\eta:=\esssup_{\tau\in\mathcal{S}_0} X(\tau)\in \textrm{Dom}(\mathcal{E})$. Under this new condition, since $0\leq v(S)\leq \mathcal{E}_S[\eta]$, it follows that $v(S)\in\textrm{Dom}^+(\mathcal{E})$.
	\end{remark}

	The following proposition gives the characterization of the optimal stopping time for the value function \eqref{1.1}.
	
	\begin{proposition}\label{P2}
		For each fixed $S\in\mathcal{S}_0$, let $\tau^*\in\mathcal{S}_S$ be such that $\mathcal{E}[X(\tau^*)]<\infty$. The following statements are equivalent:
		\begin{description}
			\item[(a)] $\tau^*$ is $S$-optimal for $v(S)$, i.e.,
			\begin{equation}\label{1.2}
			v(S)=\mathcal{E}_S[X(\tau^*)];
			\end{equation}
			\item[(b)] $v(\tau^*)=X(\tau^*)$ and $\mathcal{E}[v(S)]=\mathcal{E}[v(\tau^*)]$;
			\item[(c)] $\mathcal{E}[v(S)]=\mathcal{E}[X(\tau^*)]$.
		\end{description}
	\end{proposition}
	
	\begin{proof}
		(a)$\Rightarrow$(b): By Proposition \ref{P1}, we have $v(\tau^*)\geq X(\tau^*)$ and
		\begin{displaymath}
		v(S)\geq \mathcal{E}_S[v(\tau^*)]\geq \mathcal{E}_S[X(\tau^*)]=v(S).
		\end{displaymath}
		Taking $\mathcal{E}$-expectation on both sides implies that $\mathcal{E}[v(S)]=\mathcal{E}[v(\tau^*)]$. If $P(v(\tau^*)>X(\tau^*))>0$, by the strict monotonicity, we get that $\mathcal{E}_S[v(\tau^*)]> \mathcal{E}_S[X(\tau^*)]$, which leads to a contradiction. Therefore, we have $v(\tau^*)=X(\tau^*)$.
		
		(b)$\Rightarrow$(c): This conclusion is trivial.
		
		(c)$\Rightarrow$(a): Since $v(S)\geq \mathcal{E}_S[X(\tau^*)]$, we assume that $P(v(S)\geq \mathcal{E}_S[X(\tau^*)])>0$. Again by the strict monotonicity, we have $\mathcal{E}[v(S)]>\mathcal{E}[\mathcal{E}_S[X(\tau^*)]]=\mathcal{E}[X(\tau^*)]$, which is a contradiction.
	\end{proof}

     \begin{remark}\label{r2}
    	It is worth mentioning that most of the results in Proposition \ref{P1} and \ref{P2} still hold if the reward family is not ``adapted", which means that $X(\tau)$ is $\mathcal{F}_T$-measurable rather than $\mathcal{F}_\tau$-measurable for any $\tau\in\mathcal{S}_0$. In fact, the first difference is that $\{v(S),S\in\mathcal{S}_0\}$ is the smallest $\mathcal{E}$-supermartingale system which is greater than $\{\mathcal{E}_S[X(S)],S\in\mathcal{S}_0\}$. The second is that we need to replace $X(\tau^*)$ by $\mathcal{E}_{\tau^*}[X(\tau^*)]$ in assertion (b) of Proposition \ref{P2}. Besides, all the results do not depend on the regularity of the reward family.
    \end{remark}
	
	Now, we study the regularity of the value functions $\{v(\tau),\tau\in\mathcal{S}_0\}$ after introducing the following definition on continuity.
	\begin{definition}\label{d}
		An admissible family $\{X(\tau),\tau\in\mathcal{S}_0\}$ is said to be right-continuous (resp. left-continuous)  along stopping times in $\mathcal{E}$-expectation [RC$\mathcal{E}$ (resp., LC$\mathcal{E}$)] if for any $\tau\in\mathcal{S}_0$ and $\{\tau_n\}_{n\in\mathbb{N}}\subset\mathcal{S}_0$ such that $\tau_n\downarrow \tau$ a.s. (resp., $\tau_n\uparrow\tau$ a.s.), we have $\mathcal{E}[X(\tau)]=\lim_{n\rightarrow\infty}\mathcal{E}[X(\tau_n)]$. The family $\{X(\tau),\tau\in\mathcal{S}_0\}$ is called continuous along stopping times in $\mathcal{E}$-expectation (C$\mathcal{E}$) if it is both RC$\mathcal{E}$ and LC$\mathcal{E}$.
	\end{definition}
	
	
	\begin{proposition}\label{P3}
		Suppose the admissible family $\{X(\tau),\tau\in\mathcal{S}_0\}$ is RC$\mathcal{E}$ with $\sup_{\tau\in\mathcal{S}_0}\mathcal{E}[X(\tau)]<\infty$. Then the family $\{v(\tau),\tau\in\mathcal{S}_0\}$ is RC$\mathcal{E}$.
	\end{proposition}
	
	\begin{proof}
		Otherwise, assume that $\{v(\tau),\tau\in\mathcal{S}_0\}$ is not RC$\mathcal{E}$ at some $S\in\mathcal{S}_0$. Noting that $\{v(\tau),\tau\in\mathcal{S}_0\}$ is a supermartingale system, it follows that $\mathcal{E}[v(\tau)]\leq \mathcal{E}[v(\sigma)]$ for any $\tau\geq \sigma$. Then, there exist a constant $\varepsilon>0$ and a sequence of stopping times $\{\tau_n\}_{n\in\mathbb{N}}$ with $\tau_n\downarrow S$, such that
		\begin{equation}\label{1.4}
		\sup_{n\in\mathbb{N}}\mathcal{E}[v(\tau_n)]+\varepsilon\leq \mathcal{E}[v(S)].
		\end{equation}
		Recalling Equation \eqref{1.3}, there exists some $\sigma\in\mathcal{S}_S$, such that
		\begin{displaymath}
		\sup_{n\in\mathbb{N}}\sup_{\tau\in\mathcal{S}_{\tau_n}}\mathcal{E}[X(\tau)]+\frac{\varepsilon}{2}\leq \mathcal{E}[X(\sigma)],
		\end{displaymath}
		which implies that for any $n\in\mathbb{N}$,
		\begin{displaymath}
		\mathcal{E}[X(\sigma\vee\tau_n)]+\frac{\varepsilon}{2}\leq \mathcal{E}[X(\sigma)].
		\end{displaymath}
		Since $(\sigma\vee\tau_n)\downarrow \sigma$, by the RC$\mathcal{E}$ property of $X$, we deduce that $\mathcal{E}[X(\sigma)]+\frac{\varepsilon}{2}\leq \mathcal{E}[X(\sigma)]$, which is a contradiction.
	\end{proof}

    \begin{remark}\label{r1}
    (i)	By Remark \ref{r2}, the above result does not rely on the ``adapted'' property of the reward family.
    	
    (ii)	For any fixed $\sigma\in\mathcal{S}_0$, suppose that the admissible family $\{X(\tau),\tau\in\mathcal{S}_0\}$  is right-continuous in $\mathcal{E}$-expectation along all stopping times greater than $\sigma$, which means that if $S\in\mathcal{S}_\sigma$ and $\{S_n\}_{n\in\mathbb{N}}\subset\mathcal{S}_\sigma$ satisfy $S_n\downarrow S$, then we have $\lim_{n\rightarrow\infty}\mathcal{E}[X(S_n)]=\mathcal{E}[X(S)]$. Following the proof of Proposition \ref{P3}, the family $\{v(\tau),\tau\in\mathcal{S}_0\}$ is right-continuous in $\mathcal{E}$-expectation along all stopping times greater than $\sigma$.
    	
    (iii)	Furthermore, if the RC$\mathcal{E}$ admissible family $\{X(\tau),\tau\in\mathcal{S}_\sigma\}$ is only well-defined for the stopping times greater than $\sigma$, by a similar analysis as the proof of Proposition \ref{P1}, $\{v(S),S\in\mathcal{S}_0\}$ is still an $\mathcal{E}$-supermartingale system but without the dominance that $v(S)\geq X(S)$ for $S\leq \sigma$. Again following the proof of Proposition \ref{P3}, the family $\{v(\tau),\tau\in\mathcal{S}_0\}$ is right-continuous in $\mathcal{E}$-expectation along all stopping times greater than $\sigma$.
    \end{remark}
	
	In order to show the existence of the optimal stopping time for the value function $v(S)$, we need to furthermore assume that the $\mathbb{F}$-expectation $(\mathcal{E},\textrm{Dom}(\mathcal{E}))$ satisfies the following conditions:
	\begin{description}
			\item[(H6)] ``Sub-additivity": for any $\tau\in\mathcal{S}_0$ and $\xi,\eta\in\textrm{Dom}^+(\mathcal{E})$, $\mathcal{E}_\tau[\xi+\eta]\leq \mathcal{E}_\tau[\xi]+\mathcal{E}_\tau[\eta]$;
		\item[(H7)] ``Positive homogeneity": for any $\tau\in\mathcal{S}_0$, $\lambda\geq 0$ and $\xi\in \textrm{Dom}^+(\mathcal{E})$, $\mathcal{E}_\tau[\lambda \xi]=\lambda\mathcal{E}_\tau[\xi]$.
	\end{description}
	
	The main idea to prove the existence is applying an approximation method. More precisely, for $\lambda\in(0,1)$, we define an $\mathcal{F}_S$-measurable random variable $\tau^\lambda(S)$ by
	\begin{equation}\label{1.6}
	\tau^\lambda(S)=\essinf\{\tau\in\mathcal{S}_S: \lambda v(\tau)\leq X(\tau),\textrm{ a.s.}\}.
	\end{equation}
	We will show that the sequence $\{\tau^\lambda(S)\}_{\lambda\in(0,1)}$ admits a limit as $\lambda$ goes to $1$ and the limit is the optimal stopping time. Our first observation is that the stopping time $\tau^\lambda(S)$ is $(1-\lambda)$-optimal for the problem \eqref{1.3}.
	\begin{lemma}\label{L1}
		Let the $\mathbb{F}$-expectation $(\mathcal{E},\textrm{Dom}(\mathcal{E}))$ satisfy all the assumptions (H0)-(H7) and suppose that $\{X(\tau),\tau\in\mathcal{S}_0\}$ is a C$\mathcal{E}$ admissible family with $\sup_{\tau\in\mathcal{S}_0}\mathcal{E}[X(\tau)]<\infty$. For each $S\in\mathcal{S}_0$ and $\lambda\in(0,1)$, the stopping time $\tau^\lambda(S)$ satisfies that
		\begin{equation}\label{1.8}
		\lambda \mathcal{E}[v(S)]\leq \mathcal{E}[X(\tau^\lambda(S))].
		\end{equation}
	\end{lemma}

    To prove this lemma, we first need the following lemma.
    
    	\begin{lemma}\label{L2}
    	Under the same assumptions as Lemma \ref{L1}, for each $S\in\mathcal{S}_0$ and $\lambda\in(0,1)$, the stopping time $\tau^\lambda(S)$ satisfies that
    	\begin{equation}\label{1.11}
    	v(S)=\mathcal{E}_S[v(\tau^\lambda(S))].
    	\end{equation}
    \end{lemma}

    	\begin{proof}
    	For simplicity, we denote $\mathcal{E}_S[v(\tau^\lambda(S))]$ by $J^\lambda(S)$. Recalling that $\{v(\tau),\tau\in\mathcal{S}_0\}$ is an $\mathcal{E}$-supermartingale system, we have $J^\lambda(S)\leq v(S)$. It remains to prove the reverse inequality.
    	
    	We first claim that $\{J^\lambda(\tau),\tau\in\mathcal{S}_0\}$ is an $\mathcal{E}$-supermartingale system. Indeed, let $S,S'\in\mathcal{S}_0$ be such that $S\leq S'$. Noting that $\{v(\tau),\tau\in\mathcal{S}_0\}$ is an $\mathcal{E}$-supermartingale system, it is easy to check that $S\leq \tau^\lambda(S)\leq \tau^\lambda(S')$ and
    	\begin{displaymath}
    	\mathcal{E}_S[J^\lambda(S')]=\mathcal{E}_S[v(\tau^\lambda(S'))]=\mathcal{E}_S[\mathcal{E}_{\tau^\lambda(S)}[v(\tau^\lambda(S'))]]\leq \mathcal{E}_S[v(\tau^\lambda(S))]=J^\lambda(S).
    	\end{displaymath}
    	Hence, the claim holds.
    	
    	We then show that for any $S\in\mathcal{S}_0$ and $\lambda\in(0,1)$, we have $L^\lambda(S)\geq X(S)$, where $L^\lambda(S)=\lambda v(S)+(1-\lambda)J^\lambda(S)$. Indeed, by simple calculation, we obtain that
    	\begin{align*}
    	L^\lambda(S)=&\lambda v(S)+(1-\lambda)J^\lambda(S)=\lambda v(S)+(1-\lambda)J^\lambda(S)I_{\{\tau^\lambda(S)=S\}}+(1-\lambda)J^\lambda(S)I_{\{\tau^\lambda(S)>S\}}\\
    	=&\lambda v(S)+(1-\lambda)\mathcal{E}_S[v(\tau^\lambda(S))]I_{\{\tau^\lambda(S)=S\}}+(1-\lambda)J^\lambda(S)I_{\{\tau^\lambda(S)>S\}}\\
    	=&\lambda v(S)+(1-\lambda)v(S)I_{\{\tau^\lambda(S)=S\}}+(1-\lambda)J^\lambda(S)I_{\{\tau^\lambda(S)>S\}}\\
    	\geq & v(S)I_{\{\tau^\lambda(S)=S\}}+\lambda v(S)I_{\{\tau^\lambda(S)>S\}}\geq  X(S)I_{\{\tau^\lambda(S)=S\}}+X(S)I_{\{\tau^\lambda(S)>S\}}=X(S),		
    	\end{align*}
    	where in the first inequality we used that $J^\lambda(S)\geq 0$ and the last inequality follows from $v(S)\geq X(S)$ and the definition of $\tau^\lambda(S)$.
    	
    	Since the $\mathbb{F}$-expectation $(\mathcal{E},\textrm{Dom}(\mathcal{E}))$ satisfies (H6) and (H7), it is easy to check that $\{L^\lambda(\tau),\tau\in\mathcal{S}_0\}$ is an $\mathcal{E}$-supermartingale system. By Proposition \ref{P1}, we have $L^\lambda(S)\geq v(S)$, which, together with $v(S)<\infty$ obtained in Remark \ref{R1} implies that $J^\lambda(S)\geq v(S)$. The above analysis completes the proof.
    \end{proof}
	
	\begin{proof}[Proof of Lemma \ref{L1}]
		Fix $S\in\mathcal{S}_0$ and $\lambda\in(0,1)$. For any $\tau^i\in\mathcal{S}_S$, such that $\lambda v(\tau^i)\leq X(\tau^i)$, $i=1,2$, it is easy to check that the stopping time $\tau$ defined by $\tau=\tau^1\wedge \tau^2$ preserves the same property as $\tau^i$. Hence, there exists a sequence of stopping times $\{\tau_n\}_{n\in\mathbb{N}}\subset \mathcal{S}_S$ with $\lambda v(\tau_n)\leq X(\tau_n)$, such that $\tau_n\downarrow \tau^\lambda(S)$. By the monotonicity and positive homogeneity, we have  $\lambda\mathcal{E}[v(\tau_n)]\leq \mathcal{E}[X(\tau_n)]$ for any $n\in\mathbb{N}$. Letting $n$ go to infinity and applying the RC$\mathcal{E}$ property of $v$ and $X$ yield that
		\begin{equation}\label{1.10}
		\lambda \mathcal{E}[v(\tau^\lambda(S))]\leq \mathcal{E}[X(\tau^\lambda(S))].
		\end{equation}
		

		Now combining Equation \eqref{1.10} and \eqref{1.11}, we obtain that
		\begin{displaymath}
		\lambda \mathcal{E}[v(S)]=\lambda \mathcal{E}[v(\tau^\lambda(S))]\leq \mathcal{E}[X(\tau^\lambda(S))].
		\end{displaymath}
		The proof is complete.
	\end{proof}
	

	\begin{theorem}\label{T1}
		Under the same assumptions as Lemma \ref{L1}, for each $S\in\mathcal{S}_0$, there exists an optimal stopping time for $v(S)$ defined by \eqref{1.1}. Furthermore, the following stopping time 
		\begin{equation}\label{1.5}
		\tau^*(S)=\essinf\{\tau\in\mathcal{S}_S: v(\tau)=X(\tau) \textrm{ a.s.}\}
		\end{equation}	
		is the minimal optimal stopping time for $v(S)$.
	\end{theorem}
	
	\begin{proof}
		Observe that the stopping time $\tau^\lambda(S)$ defined by \eqref{1.6} is nondecreasing in $\lambda$ for any $S\in\mathcal{S}$. Define
		\begin{equation}\label{1.9}
		\hat{\tau}(S):=\lim_{\lambda\uparrow 1} \tau^\lambda(S).
		\end{equation}
		We claim that $\hat{\tau}(S)$ is optimal for the value function $v(S)$. In fact, letting $\lambda$ converge monotonically up to $1$ in Equation \eqref{1.8}, by the LC$\mathcal{E}$ property of $X$, we have $\mathcal{E}[v(S)]\leq \mathcal{E}[X(\hat{\tau}(S))]$. Recalling Equation \eqref{1.3}, we obtain the reverse inequality. Therefore, $\mathcal{E}[v(S)]= \mathcal{E}[X(\hat{\tau}(S))]$. Applying Proposition \ref{P2}, we get the desired result.
		
		Now we prove that $\tau^*(S)$ defined by \eqref{1.5} is the minimal optimal stopping time for $v(S)$. For simplicity, set $\mathbb{T}_S=\{\tau\in\mathcal{S}_S: v(\tau)=X(\tau) \textrm{ a.s.}\}$. By a similar analysis as the proof of Lemma \ref{L1}, there exists a sequence of stopping times $\{\tau_n\}_{n\in\mathbb{N}}\subset \mathbb{T}_S$ such that $\tau_n\downarrow \tau^*(S)$, which implies that $\tau^*(S)$ is a stopping time.
		
		Let $\tau$ be an optimal stopping time for $v(S)$. Applying Proposition \ref{P2}, we have $v(\tau)=X(\tau)$. Therefore, we derive that
		\begin{displaymath}
		\tau^*(S)\leq \essinf\{\tau\in\mathcal{S}_S: \tau \textrm{ is optimal for }v(S)\}.
		\end{displaymath}
		On the other hand, by the definition of $\tau^\lambda(S)$, it is easy to check that $\tau^\lambda(S)\leq \tau^*(S)$. It follows that
		\begin{displaymath}
		\tau^*(S)\geq \lim_{\lambda\uparrow 1}\tau^\lambda(S)=\hat{\tau}(S).
		\end{displaymath}
		Noting that $\hat{\tau}(S)$ is optimal for $v(S)$, we deduce that
		\begin{displaymath}
		\tau^*(S)\geq \essinf\{\tau\in\mathcal{S}_S: \tau \textrm{ is optimal for }v(S)\}.
		\end{displaymath}
		Hence, the above inequality turns to be an equality. The proof is complete.
	\end{proof}
	
	\begin{remark}
		Compared with the usual case that the optimal stopping time is defined trajectorially, our optimal stopping time is interpreted as the essential infimum, which makes it possible to relax the condition on the regularity of the reward family. For example, in \cite{CR}, the reward $\{X_t\}_{t\in[0,T]}$ is assumed to be RCLL and LC$\mathcal{E}$. The price for the weak condition of regularity is that the $\mathbb{F}$-expectation $(\mathcal{E},\textrm{Dom}(\mathcal{E}))$ should be positive homogenous and sub-additive. These two assumptions on the $\mathbb{F}$-expectation are mainly used to prove the existence of the optimal stopping time. For the properties which do not depend on the existence of the optimal stopping time, we may drop the positive homogeneity and sub-additivity on the $\mathbb{F}$-expectation.
	\end{remark}
	
	With the help of the existence of the optimal stopping time, we may establish the LC$\mathcal{E}$ property of the value function when the reward family is LC$\mathcal{E}$.
	\begin{proposition}\label{P4}
		Under the same assumption with Theorem \ref{T1}, the value function $\{v(\tau),\tau\in\mathcal{S}_0\}$ is LC$\mathcal{E}$.
	\end{proposition}
	
	\begin{proof}
		Let $\{S_n\}_{n\in\mathbb{N}}\subset \mathcal{S}_0$ be such that $S_n\uparrow S$, where $S$ is a stopping time. Since $\{v(\tau),\tau\in\mathcal{S}_0\}$ is an $\mathcal{E}$-supermartingale system, we have
		\begin{equation}\label{1.12}
		\mathcal{E}[v(S_n)]\geq \mathcal{E}[v(S)].
		\end{equation}
		Set
		\begin{displaymath}
		\tau^*(S_n)=\essinf\{\tau\in\mathcal{S}_{S_n}: v(\tau)=X(\tau)\}.
		\end{displaymath}
		Applying Theorem \ref{T1} implies that $v(S_n)=\mathcal{E}_{S_n}[X(\tau^*(S_n))]$. It is easy to check that $\tau^*(S_n)$ is nondecreasing in $n$. We denote the limit of $\tau^*(S_n)$ by $\bar{\tau}$. Since for any $n\in\mathbb{N}$, we have $S_n\leq \tau^*(S_n)\leq \tau^*(S)$. It follows that $S\leq \bar{\tau}\leq \tau^*(S)$. Noting that $X$ is LC$\mathcal{E}$, we have
		\begin{displaymath}
		\mathcal{E}[v(S)]=\sup_{\tau\in\mathcal{S}_S}\mathcal{E}[X(\tau)]\geq \mathcal{E}[X(\bar{\tau})]=\lim_{n\rightarrow\infty}\mathcal{E}[X(\tau^*(S_n))]=\lim_{n\rightarrow\infty}\mathcal{E}[v(S_n)],
		\end{displaymath}
		which, together with \eqref{1.12} implies the desired result.
	\end{proof}

\begin{remark}\label{r4}
	(i) In the proof of Proposition \ref{P4}, we have $\mathcal{E}[v(S)]=\mathcal{E}[X(\bar{\tau})]$. By Proposition \ref{P2}, we conclude that $v(\bar{\tau})=X(\bar{\tau})$, which implies that $\bar{\tau}$ is no less than $\tau^*(S)$ defined in Theorem \ref{T1}. On the other hand, it is easy to check that $\tau^*(S_n)\leq \tau^*(S)$ for any $n\in\mathbb{N}$. Letting $n$ go to infinity yields that $\bar{\tau}\leq \tau^*(S)$. The above analysis shows that
	\begin{displaymath}
		\tau^*(S)=\lim_{n\rightarrow\infty}\tau^*(S_n),
	\end{displaymath}
	that is, the mapping $S\mapsto \tau^*(S)$ is left-continuous along stopping times.
	
	(ii) Suppose that the family $\{X(\tau),\tau\in\mathcal{S}_0\}$ in Proposition \ref{P4} is only left-continuous in $\mathcal{E}$-expectation along stopping times greater than $\sigma$ (i.e., if $\{\tau_n\}_{n\in\mathbb{N}}\subset \mathcal{S}_\sigma$ and $\tau_n\uparrow\tau$, then we have $\mathcal{E}[X(\tau)]=\lim_{n\rightarrow\infty}\mathcal{E}[X(\tau_n)]$). If for any $S\in\mathcal{S}_0$, the optimal stopping time $\tau^*(S)$ defined by \eqref{1.5} is no less than $\sigma$, the value function $\{v(S),S\in\mathcal{S}_0\}$ is still LC$\mathcal{E}$.
\end{remark}

In the following of this section, suppose that the $\mathbb{F}$-expectation $(\mathcal{E},\textrm{Dom}(\mathcal{E}))$ satisfies all the assumptions (H0)-(H7). Now given an admissible family $\{X(\tau),\tau\in\mathcal{S}_0\}$ with $\sup_{\tau\in\mathcal{S}_0}\mathcal{E}[X(\tau)]<\infty$, for each fixed $\theta\in\mathcal{S}_0$, we define the following random variable:
\begin{displaymath}
	X'(\tau)=X(\tau)I_{\{\tau\geq \theta\}}-I_{\{\tau<\theta\}}.
\end{displaymath}
Then, for each $\tau\in\mathcal{S}_0$, $X'(\tau)$ is $\mathcal{F}_\tau$-measurable, bounded from below and \[\sup_{\tau\in\mathcal{S}_0}\mathcal{E}[|X'(\tau)|]<\infty.\] Besides, $X'(\tau)=X'(\sigma)$ on the set $\{\tau=\sigma\}$. Let us define
\begin{displaymath}
	v'(S)=\esssup_{\tau\in\mathcal{S}_S}\mathcal{E}_S[X'(\tau)].
\end{displaymath}
Then all the results in Proposition \ref{P1}, Proposition \ref{P2} and Remark \ref{R1} still hold if we replace $X$ and $v$ by $X'$ and $v'$ respectively. Furthermore, if the original admissible family $\{X(\tau),\tau\in\mathcal{S}_0\}$ is RC$\mathcal{E}$, by Remark \ref{r1}, the family $\{v'(S),S\in\mathcal{S}_0\}$ is right-continuous in $\mathcal{E}$-expectation along stopping times greater than $\theta$. The following theorem indicates that there exists an optimal stopping time for $v'(S)$ and the family $\{v'(\tau),\tau\in\mathcal{S}_0\}$ is LC$\mathcal{E}$ (not only left-continuous in $\mathcal{E}$-expectation along stopping times greater than $\theta$) provided that the family $\{X(\tau),\tau\in\mathcal{S}_0\}$ is C$\mathcal{E}$.

\begin{theorem}\label{t1}
	Let the $\mathbb{F}$-expectation $(\mathcal{E},\textrm{Dom}(\mathcal{E}))$ satisfy all the assumptions (H0)-(H7) and let $\{X(\tau),\tau\in\mathcal{S}_0\}$ be a C$\mathcal{E}$ admissible family with $\sup_{\tau\in\mathcal{S}_0}\mathcal{E}[X(\tau)]<\infty$. For each $S\in\mathcal{S}_0$, there exists an optimal stopping time for $v'(S)$. Furthermore, the following stopping time
	\begin{displaymath}
		\tau'(S)=\essinf\{\tau\in\mathcal{S}_S: v'(\tau)=X'(\tau) \textrm{ a.s.}\}
	\end{displaymath}
	is the minimal optimal stopping time for $v'(S)$ and the value function $\{v'(S),S\in\mathcal{S}_0\}$ is LC$\mathcal{E}$.
\end{theorem}

\begin{proof}
	For any $\lambda\in(0,1)$, we define a random variable $\tau^{\prime,\lambda}(S)$ by
	\begin{displaymath}
		\tau^{\prime,\lambda}(S)=\essinf\{\tau\in \mathcal{S}_S: \lambda v'(\tau)\leq X'(\tau), \textrm{ a.s.}\}.
	\end{displaymath}
	Since for any $S\in\mathcal{S}_0$, we have $v'(S)\geq \mathcal{E}_S[X(T)]\geq 0$, which implies that $\tau\geq \theta$, where $\tau\in\{\tau\in \mathcal{S}_S: \lambda v'(\tau)\leq X'(\tau), \textrm{ a.s.}\}$. Therefore, we obtain that  $\tau^{\prime,\lambda}(S)\geq \theta$. It follows that for any fixed $S\in\mathcal{S}_0$ and any $\tau\geq \tau^{\prime,\lambda}(S)$, we have $X'(\tau)=X(\tau)$. Modifying the   proofs of Lemma \ref{L1}, Lemma \ref{L2}, Theorem \ref{T1} and Proposition \ref{P4}, we finally get the desired result.
\end{proof}
	
\section{The optimal double stopping problem under nonlinear expectation}	

In this section, we consider the optimal double stopping problem under the $\mathbb{F}$-expectation satisfying Assumptions (H0)-(H5). We first introduce the definition of appropriate reward family.
	\begin{definition}
		The family $\{X(\tau,\sigma),\tau,\sigma\in\mathcal{S}_0\}$ is said to be biadmissible if it satisfies the following properties:
		\begin{description}
			\item[(1)] for all $\tau,\sigma\in\mathcal{S}_0$, $X(\tau,\sigma)\in \textrm{Dom}^+_{\tau\vee\sigma}(\mathcal{E})$;
			\item[(2)] for all $\tau,\sigma,\tau',\sigma'\in\mathcal{S}_0$, $X(\tau,\sigma)=X(\tau',\sigma')$ on the set $\{\tau=\tau'\}\cap \{\sigma=\sigma'\}$.
		\end{description}
	\end{definition}
	
	
	Now, we are given a biadmissible reward family $\{X(\tau,\sigma),\tau,\sigma\in\mathcal{S}_0\}$ with $\sup_{\tau,\sigma\in\mathcal{S}_0}\mathcal{E}[X(\tau,\sigma)]<\infty$. Then, the corresponding value function is defined as follows:
	\begin{equation}\label{2.1}
	v(S)=\esssup_{\tau_1,\tau_2\in\mathcal{S}_S}\mathcal{E}_S[X(\tau_1,\tau_2)].
	\end{equation}
	
	Similar with the single optimal stopping problem, we have the following properites.
	
	\begin{proposition}\label{P5}
		If $\{X(\tau,\sigma),\tau,\sigma\in\mathcal{S}_0\}$ is a biadmissible family with $\sup_{\tau_1,\tau_2\in\mathcal{S}_0}\mathcal{E}[X(\tau_1,\tau_2)]<\infty$, then the value function $\{v(S),S\in\mathcal{S}_0\}$ defined by \eqref{2.1} satisfies the following properties:
		\begin{description}
			\item[(i)] for each $S\in\mathcal{S}_0$, there exists a sequence of pairs of stopping times $\{(\tau_1^n,\tau_2^n)\}_{n\in\mathbb{N}}\subset \mathcal{S}_S\times\mathcal{S}_S$ such that $\mathcal{E}_S[X(\tau_1^n,\tau_2^n)]$ converges monotonically up to $v(S)$;
			\item[(ii)] $\{v(S),S\in\mathcal{S}_0\}$ is an admissible family;
			\item[(iii)] $\{v(S),S\in\mathcal{S}_0\}$ is an $\mathcal{E}$-supermartingale system;
			\item[(iv)] for each $S\in\mathcal{S}_0$, we have
			\begin{displaymath}
			\mathcal{E}[v(S)]=\sup_{\tau,\sigma\in\mathcal{S}_S}\mathcal{E}[X(\tau,\sigma)].
			\end{displaymath}
		\end{description}
	\end{proposition}
	
	\begin{proof}
		(i) It is sufficient to show that the set $\{\mathcal{E}_S[X(\tau,\sigma)],\tau,\sigma\in\mathcal{S}_S\}$ is upward directed. Indeed, for any $\tau_i,\sigma_i\in\mathcal{S}_S$, $i=1,2$, set $B=\{\mathcal{E}_S[X(\tau_1,\sigma_1)]\geq \mathcal{E}_S[X(\tau_2,\sigma_2)]\}$, $\tau=\tau_1 I_B+\tau_2I_{B^c}$ and $\sigma=\sigma_1 I_B+\sigma_2I_{B^c}$. It is easy to check that $B\in\mathcal{F}_S$, $\tau,\sigma\in\mathcal{S}_S$ and
		\begin{displaymath}
		\mathcal{E}_S[X(\tau,\sigma)]=\mathcal{E}_S[X(\tau_1,\sigma_1)]I_B+\mathcal{E}_S[X(\tau_2,\sigma_2)]I_{B^c}=\max\{\mathcal{E}_S[X(\tau_1,\sigma_1)],\mathcal{E}_S[X(\tau_2,\sigma_2)]\}.
		\end{displaymath}
		Hence, we get the desired result.
		
		(ii) The measurability and nonnegativity follow from the definition of $v(S)$. By (i), we have $v(S)=\lim_{n\rightarrow\infty}\mathcal{E}_S[X(\tau_1^n,\tau_2^n)]$. Due to the fact that
		\begin{displaymath}
			\liminf_{n\rightarrow\infty}\mathcal{E}[\mathcal{E}_S[X(\tau_1^n,\tau_2^n)]]\leq \sup_{\tau_1,\tau_2\in\mathcal{S}_0}\mathcal{E}[X(\tau_1,\tau_2)]<\infty,
		\end{displaymath}
		then Assumption (H5) implies that $v(S)\in \textrm{Dom}^+(\mathcal{E})$.  
		Given $\tau,\sigma\in \mathcal{S}_0$, we define $A=\{\tau=\sigma\}\in\mathcal{F}_\tau\cap \mathcal{F}_\sigma$. For any $\tau_1,\tau_2\in\mathcal{S}_\tau$, set $\tau_i^A=\tau_i I_A+TI_{A^c}$, $i=1,2$. Then, we have $\tau_i^A\in\mathcal{S}_\tau\cap\mathcal{S}_\sigma$, $i=1,2$. By simple calculation, we obtain that
		\begin{align*}
		\mathcal{E}_\tau[X(\tau_1,\tau_2)]I_A&=\mathcal{E}_\tau[X(\tau_1,\tau_2)I_A]=\mathcal{E}_\tau[X(\tau_1^A,\tau_2^A)I_A]=\mathcal{E}_\tau[X(\tau_1^A,\tau_2^A)]I_A\\ &=\mathcal{E}_\sigma[X(\tau_1^A,\tau_2^A)]I_A\leq v(\sigma) I_A.
		\end{align*}
		Taking essential supremum over all $\tau_1,\tau_2\in\mathcal{S}_\tau$ yields that $v(\tau)\leq v(\sigma)$ on the set $A$. By symmetry, the reverse inequality also holds. Therefore, $\{v(S),S\in\mathcal{S}_0\}$ is an admissible family.
		
		
		Properties (iii) and (iv)  can be proved similarly as the single stopping problem (see Proposition \ref{P1}).
	\end{proof}
	
	\begin{remark}\label{R2}
		(i) 
		Under the integrability condition $\sup_{\tau,\sigma\in\mathcal{S}_0}\mathcal{E}[X(\tau,\sigma)]<\infty$ and (iv) in Proposition \ref{P5}, we conclude that $v(S)<\infty$, a.s.
		
		(ii) If Assumption (H5) does not hold, we need to assume furthermore that $\eta:=\esssup_{\tau,\sigma\in\mathcal{S}_0}X(\tau,\sigma)\in\textrm{Dom}(\mathcal{E})$ in order to ensure that Proposition \ref{P5} still holds.
	\end{remark}
	
	In the following, we will show that the value function defined by \eqref{2.1} coincides with the value function of the single stopping problem corresponding to a new reward family. For this purpose,  for each $\tau\in\mathcal{S}_0$, we define 	
	\begin{equation}\label{2.2}
	u_1(\tau)=\esssup_{\tau_1\in\mathcal{S}_\tau}\mathcal{E}_\tau[X(\tau_1,\tau)], \ u_2(\tau)=\esssup_{\tau_2\in\mathcal{S}_\tau}\mathcal{E}_\tau[X(\tau,\tau_2)],
	\end{equation}
	and
	\begin{equation}\label{2.3}
	\widetilde{X}(\tau)=\max\{u_1(\tau),u_2(\tau)\}.
	\end{equation}
	The first observation is that, the family $\{\widetilde{X}(\tau),\tau\in\mathcal{S}_0\}$ is admissible.
	\begin{lemma}
		Suppose that $\{X(\tau,\sigma),\tau,\sigma\in\mathcal{S}_0\}$ is a biadmissible family with $\sup_{\tau,\sigma}\mathcal{E}[X(\tau,\sigma)]<\infty$. Then, the family defined by \eqref{2.3} is admissible and $\sup_{\tau\in\mathcal{S}_0}\mathcal{E}[\widetilde{X}(\tau)]<\infty$.
	\end{lemma}

  \begin{proof}
  	It is sufficient to prove that $\{u_1(\tau),\tau\in\mathcal{S}_0\}$ is admissible. Similar with the proof of Proposition \ref{P5}, $u_1(\tau)$ is $\mathcal{F}_\tau$-measurable and $u_1(\tau)\in\textrm{Dom}^+(\mathcal{E})$. For each fixed $\tau,\sigma\in\mathcal{S}_0$, set $A=\{\tau=\sigma\}$ and $\theta^A=\theta I_A+T I_{A^c}$, where $\theta\in\mathcal{S}_\tau$. It is easy to check that $A\in\mathcal{F}_{\tau\wedge\sigma}$,  $\theta^A\in\mathcal{S}_\sigma$ and
  	\begin{align*}
  		\mathcal{E}_\tau[X(\theta,\tau)]I_A=\mathcal{E}_\tau[X(\theta,\tau)I_A]=\mathcal{E}_\tau[X(\theta^A,\sigma)I_A]=\mathcal{E}_\tau[X(\theta^A,\sigma)]I_A=\mathcal{E}_\sigma[X(\theta^A,\sigma)]I_A\leq u_1(\sigma)I_A.
  	\end{align*}
  	Taking supremum over all $\theta\in\mathcal{S}_\tau$ implies that $u_1(\tau)\leq u_1(\sigma)$ on $A$. By symmetry, we have $u_1(\sigma)\leq u_1(\tau)$ on $A$. Therefore, $u_1(\tau)I_A=u_1(\sigma)I_A$.
  	
  	It is easy to verify that $0\leq \widetilde{X}(\tau)\leq v(\tau)$. By Proposition \ref{P5}, we have
  	\begin{displaymath}
  		\sup_{\tau\in\mathcal{S}_0}\mathcal{E}[\widetilde{X}(\tau)]\leq \sup_{\tau\in\mathcal{S}_0}\mathcal{E}[v(\tau)]\leq \sup_{\tau,\sigma\in\mathcal{S}_0}\mathcal{E}[X(\tau,\sigma)]<\infty.
  	\end{displaymath}
  \end{proof}
	
 The next theorem states that $\{v(S),S\in\mathcal{S}_0\}$ is the smallest $\mathcal{E}$-supermartingale system such that $v(S)\geq \widetilde{X}(S)$, for any $S\in\mathcal{S}_0$. In other words, $v(S)$ corresponds to the value function $u(S)$ associated with the reward family $\{\widetilde{X}(S),S\in\mathcal{S}_0\}$, where
	\begin{equation}\label{2.4}
	u(S)=\esssup_{\tau\in\mathcal{S}_S}\mathcal{E}_S[\widetilde{X}(\tau)].
	\end{equation}
	
	\begin{theorem}\label{T2}
		Let $\{X(\tau,\sigma),\tau,\sigma\in\mathcal{S}_0\}$ be a biadmissible family with $\sup_{\tau_1,\tau_2\in\mathcal{S}_0}\mathcal{E}[X(\tau_1,\tau_2)]<\infty$. Then, for each stopping time $S\in\mathcal{S}_0$, we have $v(S)=u(S)$.
	\end{theorem}
	
	\begin{proof}
		Fix $S\in\mathcal{S}_0$. Consider two stopping times $\tau_1,\tau_2\in\mathcal{S}_S$. Set $A=\{\tau_1\leq \tau_2\}\in\mathcal{F}_{\tau_1\wedge\tau_2}$. It is easy to check that on the set $A$, we have
		\begin{displaymath}
		\mathcal{E}_{\tau_1}[X(\tau_1,\tau_2)]\leq u_2(\tau_1)\leq \widetilde{X}(\tau_1\wedge\tau_2),
		\end{displaymath}
		and on the set $A^c$, we have
		\begin{displaymath}
		\mathcal{E}_{\tau_2}[X(\tau_1,\tau_2)]\leq u_1(\tau_2)\leq \widetilde{X}(\tau_1\wedge\tau_2).
		\end{displaymath}
		Applying the above results, we obtain that
		\begin{align*}
		\mathcal{E}_{S}[X(\tau_1,\tau_2)]=&\mathcal{E}_S[\mathcal{E}_{\tau_1\wedge\tau_2}[X(\tau_1,\tau_2)]]\\
		=&\mathcal{E}_S[\mathcal{E}_{\tau_1\wedge\tau_2}[X(\tau_1,\tau_2)]I_A+\mathcal{E}_{\tau_1\wedge\tau_2}[X(\tau_1,\tau_2)]I_{A^c}]\\
		=&\mathcal{E}_S[\mathcal{E}_{\tau_1}[X(\tau_1,\tau_2)]I_A+\mathcal{E}_{\tau_2}[X(\tau_1,\tau_2)]I_{A^c}]\\
		\leq &\mathcal{E}_S[\widetilde{X}(\tau_1\wedge\tau_2)]\leq u(S).
		\end{align*}
		Taking supermum over all $\tau_1,\tau_2\in\mathcal{S}_S$ implies that $v(S)\leq u(S)$.
		
		Now it remains to prove the reverse inequality. It is obvious that
		\begin{displaymath}
		v(S)\geq \esssup_{\tau_1\in\mathcal{S}_S}\mathcal{E}_S[X(\tau_1,S)]=u_1(S).
		\end{displaymath}
		Similarly, we have $v(S)\geq u_2(S)$. By Proposition \ref{P5}, $\{v(\tau),\tau\in\mathcal{S}_0\}$ is an $\mathcal{E}$-supermartingale which dominates the family $\{\widetilde{X}(\tau),\tau\in\mathcal{S}_0\}$. Since $\{u(\tau),\tau\in\mathcal{S}_0\}$ is the smallest  $\mathcal{E}$-supermartingale which is no less than $\{\widetilde{X}(\tau),\tau\in\mathcal{S}_0\}$, we finally obtain that $v(S)\geq u(S)$. The proof is complete.
	\end{proof}
	
	With the help of the characterization of the value function $v$ stated in Theorem \ref{T2}, we may construct the optimal stopping times for either the multiple problem \eqref{2.1} or the single problem \eqref{2.2}, \eqref{2.4} if we obtain the optimal stopping times for one of the problems.
	
	\begin{proposition}\label{P6}
		Fix $S\in\mathcal{S}_0$. Suppose that $(\tau_1^*,\tau_2^*)\in\mathcal{S}_S\times\mathcal{S}_S$ is optimal for $v(S)$. Then, we have
		\begin{description}
			\item[(1)] $\tau_1^*\wedge\tau_2^*$ is optimal for $u(S)$;
			\item[(2)] $\tau_1^*$ is optimal for $u_2(\tau_1^*)$ on the set $A$;
			\item[(3)] $\tau_2^*$ is optimal for $u_1*(\tau_2^*)$ on the set $A^c$,
		\end{description}
		where $A=\{\tau_1^*\leq \tau_2^*\}$. On the other hand, suppose that the stopping times $\theta^*,\theta_i^*$, $i=1,2$, satisfy the following conditions:
		\begin{description}
			\item[(i)] $\theta^*$ is optimal for $u(S)$;
			\item[(ii)] $\theta_1^*$ is optimal for $u_2(\theta^*)$;
			\item[(iii)] $\theta_2^*$ is optimal for $u_1(\theta^*)$,.
		\end{description}
		Set
		\begin{equation}\label{2.5}
		\sigma^*_1=\theta^* I_B+\theta_1^* I_{B^c},\ \sigma_2^*=\theta_2^* I_B+\theta^* I_{B^c},
		\end{equation}
		where $B=\{u_1(\theta^*)\leq u_2(\theta^*)\}$. Then, the pair $(\sigma^*_1,\sigma^*_2)$ is optimal for $v(S)$.
	\end{proposition}
	
	\begin{proof}
		We first prove the necessary condition of optimality for the multiple stopping problem. Let $(\tau_1^*,\tau_2^*)$ be optimal for $v(S)$. Since $v(S)=u(S)$, all the inequalities in the proof of Theorem \ref{T2} turn into equalities if $(\tau_1,\tau_2)$ is replaced by $(\tau_1^*,\tau_2^*)$. More precisely, we have
		\begin{equation}\begin{split}\label{2}
		v(S)&=\mathcal{E}_S[X(\tau_1^*,\tau_2^*)]=\mathcal{E}_S[\widetilde{X}(\tau_1^*\wedge\tau_2^*)]=u(S),\\
		\mathcal{E}_{\tau_1^*}[X(\tau_1^*,\tau_2^*)]&=u_2(\tau_1^*)=u_2(\tau_1^*\wedge\tau_2^*)=\widetilde{X}(\tau_1^*\wedge\tau_2^*) \textrm{ on the set } A,\\
		\mathcal{E}_{\tau_2^*}[X(\tau_1^*,\tau_2^*)]&=u_1(\tau_2^*)=u_1(\tau_1^*\wedge\tau_2^*)=\widetilde{X}(\tau_1^*\wedge\tau_2^*) \textrm{ on the set } A^c,
		\end{split}\end{equation}
		which implies the assertion (1)-(3).
		
		Now suppose that $\theta^*,\theta^*_i$ satisfy conditions (i)-(iii), $i=1,2$. Noting that $B\in\mathcal{F}_{\theta^*}$, by simple calculation, we have
		\begin{align*}
		u(S)=&\mathcal{E}_S[\widetilde{X}(\theta^*)]=\mathcal{E}_S[u_2(\theta^*)I_B+u_1(\theta^*)I_{B^c}]\\
		=&\mathcal{E}_S[\mathcal{E}_{\theta^*}[X(\theta^*,\theta^*_2)]I_B+\mathcal{E}_{\theta^*}[X(\theta^*_1,\theta^*)]I_{B^c}]\\
		=&\mathcal{E}_S[\mathcal{E}_{\theta^*}[X(\theta^*,\theta^*_2)I_B+X(\theta^*_1,\theta^*)I_{B^c}]]\\
		=&\mathcal{E}_S[X(\sigma^*_1,\sigma_2^*)].
		\end{align*}
		By Theorem \ref{T2}, it follows that $v(S)=\mathcal{E}_S[X(\sigma^*_1,\sigma_2^*)]$. The proof is complete.
	\end{proof}
	
	\begin{remark}
		By Equation \eqref{2}, it is easy to check that
		\begin{displaymath}
		A=\{\tau_1^*\leq \tau_2^*\}\subset \{u_1(\tau^*_1\wedge \tau_2^*)\leq u_2(\tau^*_1\wedge \tau_2^*)\}.
		\end{displaymath}
		Furthermore, since $\theta^*_i\in\mathcal{S}_{\theta^*}$, by the definition of $\sigma^*_i$, $i=1,2$, we have
		\begin{displaymath}
		B=\{u_1(\theta^*)\leq u_2(\theta^*)\}\subset \{\sigma^*_1\leq \sigma^*_2\}.
		\end{displaymath}
		Therefore, we conclude that
		\begin{displaymath}
		B\subset \{\sigma^*_1\leq \sigma^*_2\}\subset \{u_1(\sigma^*_1\wedge \sigma_2^*)\leq u_2(\sigma^*_1\wedge \sigma_2^*)\}=B.
		\end{displaymath}
		That is, $B=\{\sigma^*_1\leq \sigma^*_2\}= \{u_1(\theta^*)\leq u_2(\theta^*)\}$. However, the inclusion for $A$ may be strict. We may refer to Remark 2.2 in \cite{KQR}.
	\end{remark}
	
	By Proposition \ref{P6}, in order to obtain the multiple optimal stopping times for $v(S)$ defined by \eqref{2.1}, it is sufficient to derive the optimal stopping times for the auxiliary single stopping problems \eqref{2.2} and \eqref{2.4}. For this purpose, according to Theorem \ref{T1}, we need to study some regularity results for $\{\widetilde{X}(\tau),\tau\in\mathcal{S}_0\}$. Before establishing this property, we first introduce the definition of continuity for the biadmissible family.
	
	\begin{definition}
		A biadmissible family $\{X(\tau,\sigma),\tau,\sigma\in\mathcal{S}_0\}$ is said to be right-continuous (resp. left-continuous) along stopping times in $\mathcal{E}$-expectation [RC$\mathcal{E}$ (resp., LC$\mathcal{E}$)] if, for any $\tau,\sigma\in\mathcal{S}_0$ and any sequence $\{\tau_n\}_{n\in\mathbb{N}}, \{\sigma_n\}_{n\in\mathbb{N}}\subset \mathcal{S}_0$ such that $\tau_n\downarrow \tau$, $\sigma_n\downarrow\sigma$ (resp., $\tau_n\uparrow \tau$, $\sigma_n\uparrow\sigma$), one has $\mathcal{E}[X(\tau,\sigma)]=\lim_{n\rightarrow\infty}\mathcal{E}[X(\tau_n,\sigma_n)]$.
	\end{definition}
	
	By a similar proof as Proposition \ref{P3}, we have the following regularity result.
	
	\begin{proposition}\label{P7}
		If the biadmissible family $\{X(\tau,\sigma),\tau,\sigma\in\mathcal{S}_0\}$ is RC$\mathcal{E}$, then, the family $\{v(S),S\in\mathcal{S}_0\}$ defined by \eqref{2.1} is RC$\mathcal{E}$.
	\end{proposition}
	
	The regularity of the new reward family $\{\widetilde{X}(\tau),\tau\in\mathcal{S}_0\}$ requires some strong continuity of the biadmissible family. Due to the nonlinearity of the expectation, the definition is slightly different from Definition 2.3 in \cite{KQR}.
	\begin{definition}
		A biadmissible family $\{X(\tau,\sigma),\tau,\sigma\in\mathcal{S}_0\}$ is said to be uniformly right-continuous (resp. left-continuous) along stopping times in $\mathcal{E}$-expectation [URC$\mathcal{E}$ (resp., ULC$\mathcal{E}$)] if, for any $\sigma\in\mathcal{S}_0$ and any sequence $\{\sigma_n\}_{n\in\mathbb{N}}\subset \mathcal{S}_0$ such that  $\sigma_n\downarrow\sigma$ (resp., $\sigma_n\uparrow\sigma$), one has
		\begin{align*}
		&\lim_{n\rightarrow\infty}\sup_{\tau\in\mathcal{S}_0}\mathcal{E}[|X(\tau,\sigma)-X(\tau,\sigma_n)|]=0,\\&\lim_{n\rightarrow\infty}\sup_{\tau\in\mathcal{S}_0}\mathcal{E}[|X(\sigma,\tau)-X(\sigma_n,\tau)|]=0.
		\end{align*}
		Besides, the biadmissible family is said to be uniformly continuous along stopping times in $\mathcal{E}$-expectation (UC$\mathcal{E}$) if it is both URC$\mathcal{E}$ and ULC$\mathcal{E}$.
	\end{definition}
	
	\begin{definition}
		An $\mathbb{F}$-expectation $(\mathcal{E},\textrm{Dom}(\mathcal{E}))$ is said to be dominated by another $\mathbb{F}$-expectation $(\widetilde{\mathcal{E}},\textrm{Dom}(\widetilde{\mathcal{E}}))$ if $\textrm{Dom}(\mathcal{E})\subset \textrm{Dom}(\widetilde{\mathcal{E}})$ and for any $\tau\in\mathcal{S}_0$ and $\xi,\eta\in \textrm{Dom}(\mathcal{E})$, one has
		\begin{displaymath}
			\mathcal{E}_\tau[\xi+\eta]-\mathcal{E}_\tau[\eta]\leq \widetilde{\mathcal{E}}_\tau[\xi].
		\end{displaymath}
	\end{definition}

    \begin{remark}\label{R3}
    	By the requirements on the domain of $\mathcal{E}$ (see Definition \ref{D2.1} and Assumptions (H3)-(H5)), for any $\xi,\eta\in\textrm{Dom}(\mathcal{E})$, we may not conclude that $\xi-\eta\in \textrm{Dom}(\mathcal{E})$. Therefore, the above definition of dominance cannot be written as
    	\begin{displaymath}
    		\mathcal{E}_\tau[\xi]-\mathcal{E}_\tau[\eta]\leq \widetilde{\mathcal{E}}_\tau[\xi-\eta].
    	\end{displaymath}
    	However, if $(\mathcal{E},\textrm{Dom}(\mathcal{E}))$ is dominated by  $(\widetilde{\mathcal{E}},\textrm{Dom}(\widetilde{\mathcal{E}}))$, we have, for any $\tau\in\mathcal{S}_0$ and $\xi,\eta\in \textrm{Dom}(\mathcal{E})$
    	\begin{equation}\label{0}
    		|\mathcal{E}_\tau[\xi]-\mathcal{E}_\tau[\eta]|\leq \widetilde{\mathcal{E}}_\tau[|\xi-\eta|].
    	\end{equation}
    	First, if $\xi\in\textrm{Dom}(\mathcal{E})$, noting that $|\xi|=\xi I_{\{\xi\geq 0\}}$ and $\{\xi\geq 0\}\in\mathcal{F}_T$, by (D2) in Definition \ref{D2.1}, we have $|\xi|\in \textrm{Dom}(\mathcal{E})$. Since $0\leq |\xi-\eta|\leq |\xi|+|\eta|$, by (D2) and (D3), it follows that $|\xi-\eta|\in\textrm{Dom}(\mathcal{E})$. It is easy to check that
    	\begin{displaymath}
    		\mathcal{E}_\tau[\xi]-\mathcal{E}_\tau[\eta]\leq \mathcal{E}_\tau[\eta+|\xi-\eta|]-\mathcal{E}_\tau[\eta]\leq \widetilde{\mathcal{E}}_\tau[|\xi-\eta|].
    	\end{displaymath}
    	By the symmetry of $\xi$ and $\eta$, we obtain Equation \eqref{0}.
    \end{remark}

    \begin{example}
    	\begin{description}
    		\item[(1)] If the $\mathbb{F}$-expectation $(\mathcal{E},\textrm{Dom}(\mathcal{E}))$ satisfies (H6), then it is dominated by itself. Especially, $(\{E_t[\cdot]\}_{t\in[0,T]}, L^1(\mathcal{F}_T))$ is dominated by itself;
            \item[(2)] For a generator $g$ with Lipschitz constant $\kappa$, the $g$-expectation  $(\{\mathcal{E}^g_t[\cdot]\}_{t\in[0,T]},L^2(\mathcal{F}_T))$ is dominated by $(\{\mathcal{E}^{\tilde{g}}_t[\cdot]\}_{t\in[0,T]},L^2(\mathcal{F}_T))$, where $\tilde{g}(t,z)=\kappa|z|$.
            	\end{description}
    \end{example}
	
	\begin{theorem}\label{T3}
		Let $(\widetilde{\mathcal{E}},\textrm{Dom}(\widetilde{\mathcal{E}}))$ be an $\mathbb{F}$-expectation satisfying Assumptions (H0)-(H5). Suppose that the $\mathbb{F}$-expectation $(\mathcal{E},\textrm{Dom}(\mathcal{E}))$ is  dominated by $(\widetilde{\mathcal{E}},\textrm{Dom}(\widetilde{\mathcal{E}}))$ and the biadmissible family $\{X(\tau,\sigma),\tau,\sigma\in\mathcal{S}_0\}$ is URC$\widetilde{\mathcal{E}}$  with $\sup_{\tau,\sigma\in\mathcal{S}_0}\mathcal{E}[X(\tau,\sigma)]<\infty$. Then, the family $\{\widetilde{X}(\tau),\tau\in\mathcal{S}_0\}$ defined by \eqref{2.3} is RC$\mathcal{E}$.
	\end{theorem}
	
	\begin{proof}
		By the definition of $\widetilde{X}$, we only need to prove that the family $\{u_1(\tau),\tau\in\mathcal{S}_0\}$ is RC$\mathcal{E}$.  Let $\{\theta_n\}_{n\in\mathbb{N}}$ be a sequence of stopping times such that $\theta_n\downarrow\theta$. Since $\{X(\tau,\sigma),\tau,\sigma\in\mathcal{S}_0\}$ is URC$\widetilde{\mathcal{E}}$, by Equation \eqref{0}, we have 
		\begin{displaymath}
			\lim_{n\rightarrow\infty}|\mathcal{E}[X(\tau_n,\sigma)]-\mathcal{E}[X(\tau,\sigma)]|\leq \lim_{n\rightarrow\infty}\widetilde{\mathcal{E}}[|X(\tau_n,\sigma)-X(\tau,\sigma)|]=0.
		\end{displaymath}
		 It follows that for each fixed $\sigma\in\mathcal{S}_0$, the family $\{X(\tau,\sigma),\tau\in\mathcal{S}_\sigma\}$ is admissible and right-continuous in $\mathcal{E}$-expectation along stopping times greater than $\sigma$ (It is important to note that the whole family $\{X(\tau,\sigma),\tau\in\mathcal{S}_0\}$ may not be admissible since $X(\tau,\sigma)$ is $\mathcal{F}_\sigma$-measurable rather than $\mathcal{F}_\tau$-measurable if $\tau\leq \sigma$).
		 By Proposition \ref{P3} and Remark \ref{r1}, we obtain that the family $\{U_1(S,\theta),S\in\mathcal{S}_0\}$ is right-continuous in $\mathcal{E}$-expectation along stopping times greater than $\theta$, where
		 \begin{equation}\label{2.7}
		 	U_1(S,\theta)=\esssup_{\tau_1\in\mathcal{S}_{S}}\mathcal{E}_S[X(\tau_1,\theta)].
		 \end{equation}
		 That is, $\lim_{n\rightarrow\infty}\mathcal{E}[U_1(\theta_n,\theta)]=\mathcal{E}[U_1(\theta,\theta)]$. 
		
		Now, we admit the following lemma and the proof will be postponsed later.
		
		\begin{lemma}\label{L3}
			For any stopping times $\tau,\sigma_1,\sigma_2$, we have
			\begin{displaymath}
			|\mathcal{E}[U_1(\tau,\sigma_1)]-\mathcal{E}[U_1(\tau,\sigma_2)]|\leq \sup_{S\in\mathcal{S}_0}\widetilde{\mathcal{E}}[|X(S,\sigma_1)-X(S,\sigma_2)|].
			\end{displaymath}
		\end{lemma}
		Therefore, by the URC$\widetilde{\mathcal{E}}$ property of $\{X(\tau,\sigma),\tau,\sigma\in\mathcal{S}_0\}$, as $n$ goes to infinity, we obtain that
		\begin{align*}
		|\mathcal{E}[U_1(\theta_n,\theta)]-\mathcal{E}[U_1(\theta_n,\theta_n)]|\leq \sup_{S\in\mathcal{S}_0}\widetilde{\mathcal{E}}[|X(S,\theta)-X(S,\theta_n)|]\rightarrow 0.
		\end{align*}
		The above analysis indicates that
		\begin{align*}
			&\lim_{n\rightarrow\infty}|\mathcal{E}[u_1(\theta)]-\mathcal{E}[u_1(\theta_n)]|=\lim_{n\rightarrow\infty}|\mathcal{E}[U_1(\theta,\theta)]-\mathcal{E}[U_1(\theta_n,\theta_n)]|\\
			\leq &\lim_{n\rightarrow\infty}|\mathcal{E}[U_1(\theta,\theta)]-\mathcal{E}[U_1(\theta_n,\theta)]|+\lim_{n\rightarrow\infty}|\mathcal{E}[U_1(\theta_n,\theta)]-\mathcal{E}[U_1(\theta_n,\theta_n)]|=0
		\end{align*}
		The proof is complete.
	\end{proof}
	
	\begin{proof}[Proof of Lemma \ref{L3}]
		By a similar analysis as the proof of Proposition \ref{P1}, for each fixed $\tau\in\mathcal{S}_0$, there exists a sequence of stopping times $\{S_m\}_{m\in\mathbb{N}}\subset \mathcal{S}_{\tau}$ such that
		\begin{displaymath}
		\widetilde{\mathcal{E}}_{\tau}[|X(S_m,\sigma_1)-X(S_m,\sigma_2)|]\uparrow \esssup_{\tau_1\in\mathcal{S}_{\tau}}\widetilde{\mathcal{E}}_{\tau}[|X(\tau_1,\sigma_1)-X(\tau_1,\sigma_2)|]
		\end{displaymath}
		By simple calculation, we have
		\begin{align*}
		|\mathcal{E}[U_1(\tau,\sigma_1)]-\mathcal{E}[U_1(\tau,\sigma_2)]|\leq &\widetilde{\mathcal{E}}[|\esssup_{\tau_1\in\mathcal{S}_{\tau}}\mathcal{E}_{\tau}[X(\tau_1,\sigma_1)]-\esssup_{\tau_1\in\mathcal{S}_{\tau}}\mathcal{E}_{\tau}[X(\tau_1,\sigma_2)]|]\\
		\leq &\widetilde{\mathcal{E}}[\esssup_{\tau_1\in\mathcal{S}_{\tau}}|\mathcal{E}_{\tau}[X(\tau_1,\sigma_1)-X(\tau_1,\sigma_2)]|]\\
		\leq &\widetilde{\mathcal{E}}[\esssup_{\tau_1\in\mathcal{S}_{\tau}}\widetilde{\mathcal{E}}_{\tau}[|X(\tau_1,\sigma_1)-X(\tau_1,\sigma_2)|]]\\
		\leq &\liminf_{m\rightarrow\infty}\widetilde{\mathcal{E}}[|X(S_m,\sigma_1)-X(S_m,\sigma_2)|]\\
		\leq &\sup_{S\in\mathcal{S}_0}\widetilde{\mathcal{E}}[|X(S,\sigma_1)-X(S,\sigma_2)|].
		\end{align*}
		The proof is complete.
	\end{proof}

    The main difficulty is to prove the LC$\mathcal{E}$ property of the reward family $\{\widetilde{X}(\tau),\tau\in\mathcal{S}_0\}$ due to some measurability issues. More precisely,  let $\{\theta_n\}_{n\in\mathbb{N}}$ be a sequence of stopping times such that $\theta_n\uparrow\theta$. We need to prove that $\lim_{n\rightarrow\infty}\mathcal{E}[u_1(\theta_n)]=\mathcal{E}[u_1(\theta)]$. However, we cannot follow the proof of Theorem \ref{T3} of the RC$\mathcal{E}$ property. The problem is that the relation $\lim_{n\rightarrow\infty}\mathcal{E}[U_1(\theta_n,\theta)]=\mathcal{E}[U_1(\theta,\theta)]=\mathcal{E}[u_1(\theta)]$ may not hold, where $U_1$ is given by \eqref{2.7}. Although $\{U_1(S,\theta),S\in\mathcal{S}_0\}$ can be interpreted as the value function associated with the family $\{X(\tau_1, \theta),\tau_1\in\mathcal{S}_0\}$, we cannot apply Proposition \ref{P4} since the reward $\{X(\tau_1,\theta),\tau_1\in\mathcal{S}_0\}$ is not admissible. The main idea is to modify this reward slightly and then apply the LC$\mathcal{E}$ property of the modified reward family stated in Theorem \ref{t1}.

    \begin{theorem}\label{t2}
    	 Suppose that the $\mathbb{F}$-expectation $(\mathcal{E},\textrm{Dom}(\mathcal{E}))$ satisfies (H0)-(H7)  and the biadmissible family $\{X(\tau,\sigma),\tau,\sigma\in\mathcal{S}_0\}$ is UC${\mathcal{E}}$ with $\sup_{\tau,\sigma\in\mathcal{S}_0}\mathcal{E}[X(\tau,\sigma)]<\infty$. Then, the family $\{\widetilde{X}(\tau),\tau\in\mathcal{S}_0\}$ defined by \eqref{2.3} is LC$\mathcal{E}$.
    \end{theorem}

    \begin{proof}
    	By the definition of $\widetilde{X}$, it suffices to prove that $\{u_1(\tau),\tau\in\mathcal{S}_0\}$ is LC$\mathcal{E}$.  Let $\{\theta_n\}_{n\in\mathbb{N}}$ be a sequence of stopping times such that $\theta_n\uparrow\theta$. Now we define
    	\begin{displaymath}
    		X'(\tau,\theta)=X(\tau,\theta)I_{\{\tau\geq \theta\}}-I_{\{\tau<\theta\}}.
    	\end{displaymath}
    	It is easy to check that for any $\tau\in\mathcal{S}_0$, $X'(\tau,\theta)$ is $\mathcal{F}_\tau$-measurable and  bounded from below with $\sup_{\tau\in\mathcal{S}_0}\mathcal{E}[|X'(\tau,\theta)|]<\infty$. Therefore, by Theorem \ref{t1}, the value function $\{v'(S),S\in\mathcal{S}_0\}$ defined by
    	\begin{displaymath}
    		v'(S)=\esssup_{\tau\in\mathcal{S}_S}\mathcal{E}_S[X'(\tau,\theta)]
    	\end{displaymath}
    	is LC$\mathcal{E}$. It follows that $\lim_{n\rightarrow\infty}\mathcal{E}[v'(\theta_n)]=\mathcal{E}[v'(\theta)]$. By the definition of $X'$, it is easy to check that
    	\begin{align*}
    		v'(\theta)=\esssup_{\tau\in\mathcal{S}_{\theta}}\mathcal{E}_\theta[X'(\tau,\theta)]=\esssup_{\tau\in\mathcal{S}_{\theta}}\mathcal{E}_\theta[X(\tau,\theta)]=u_1(\theta),
    	\end{align*}
    	which implies that $\lim_{n\rightarrow\infty}\mathcal{E}[v'(\theta_n)]=\mathcal{E}[u_1(\theta)]$. Note that for any $\tau\in\mathcal{S}_{\theta_n}$, we have
    	\begin{align*}
    		|X'(\tau,\theta)-X(\tau,\theta_n)|=&|X(\tau,\theta)I_{\{\tau\geq \theta\}}-I_{\{\theta_n\leq \tau<\theta\}}-X(\tau,\theta_n)|\\
    		=&|X(\tau,\theta)-X(\tau,\theta_n)|I_{\{\tau\geq \theta\}}+|1+X(\tau,\theta_n)|I_{\{\theta_n\leq \tau<\theta\}}\\
    		\leq&|X(\tau,\theta)-X(\tau,\theta_n)|+|1+\esssup_{\tau,\sigma\in\mathcal{S}_0}X(\tau,\sigma)|I_{\{\theta_n<\theta\}}.
    	\end{align*}
    	Set $\eta=1+\esssup_{\tau,\sigma\in\mathcal{S}_0}X(\tau,\sigma)$. By Lemma \ref{L4.1}, we have $\eta\in\textrm{Dom}^+(\mathcal{E})$. By a similar analysis as Lemma \ref{L3}, we obtain that
    	\begin{align*}
    		|\mathcal{E}[v'(\theta_n)]-\mathcal{E}[u_1(\theta_n)]|&\leq \mathcal{E}[\esssup_{\tau\in\mathcal{S}_{\theta_n}}|X'(\tau,\theta)-X(\tau,\theta_n)|]\\
    		&\leq \mathcal{E}[\esssup_{\tau\in\mathcal{S}_{\theta_n}}|X(\tau,\theta)-X(\tau,\theta_n)|]+\mathcal{E}[\eta I_{A_n}]
    	\end{align*}
    	where $A_n=\{\theta_n<\theta\}$. For the first part of the right-hand side, it is easy to check that
    	\begin{displaymath}
    		\mathcal{E}[\esssup_{\tau\in\mathcal{S}_{\theta_n}}|X(\tau,\theta)-X(\tau,\theta_n)|]\leq \sup_{\tau\in\mathcal{S}_0}\mathcal{E}[|X(\tau,\theta)-X(\tau,\theta_n)|]\rightarrow 0, \textrm{ as }n\rightarrow \infty.
    	\end{displaymath}	
    	Noting that $I_{A_n}\downarrow 0$ and $\{A_n\}_{n\in\mathbb{N}}\subset\mathcal{F}_T$, by Assumption (H2), we obtain that $\lim_{n\rightarrow\infty}[\eta I_{A_n}]=0$. Finally, we get that
    	\begin{displaymath}
    		\lim_{n\rightarrow\infty}|\mathcal{E}[u_1(\theta)]-\mathcal{E}[u_1(\theta_n)]|\leq \lim_{n\rightarrow\infty}|\mathcal{E}[u_1(\theta)]-\mathcal{E}[v'(\theta_n)]|+\lim_{n\rightarrow\infty}|\mathcal{E}[v'(\theta_n)]-\mathcal{E}[u_1(\theta_n)]|=0.
    	\end{displaymath}
    	The proof is complete.
    \end{proof}
	
	Now, we can establish the existence of optimal stopping times for the value function defined by \eqref{2.1}.
	
	\begin{theorem}\label{T4}
	Suppose that the $\mathbb{F}$-expectation $(\mathcal{E},\textrm{Dom}(\mathcal{E}))$ satisfies (H0)-(H7)  and the biadmissible family $\{X(\tau,\sigma),\tau,\sigma\in\mathcal{S}_0\}$ is UC${\mathcal{E}}$. Then, there exists a pair of optimal stopping times $(\tau_1^*,\tau_2^*)$ for the value function $v(S)$ defined by \eqref{2.1}.
	\end{theorem}
	
	\begin{proof}
		By Theorems \ref{T1}, \ref{T3} and \ref{t2}, there exists an optimal stopping time $\theta^*$ for the value function $u(S)$ defined by \eqref{2.4}. In fact, the smallest one is given by
		\begin{displaymath}
		\theta^*=\essinf\{\theta\in\mathcal{S}_S:u(\theta)=\widetilde{X}(\theta)\}.
		\end{displaymath}
		Furthermore, the admissible families $\{X(\theta,\theta^*),\theta\in\mathcal{S}_{\theta^*}\}$ and $\{X(\theta^*,\theta),\theta\in\mathcal{S}_{\theta^*}\}$ are C$\mathcal{E}$. Let us introduce the following two  optimal single stopping problems:
		\begin{displaymath}
		v_1(S)=\esssup_{\theta\in\mathcal{S}_S}\mathcal{E}_S[X(\theta,\theta^*)], \ v_2(S)=\esssup_{\theta\in\mathcal{S}_S}\mathcal{E}_S[X(\theta^*,\theta)],
		\end{displaymath}
		where $S\in\mathcal{S}_{\theta^*}$. By Theorem \ref{T1} again, the following stopping times
		\begin{displaymath}
		\theta_1^*=\essinf\{\theta\in\mathcal{S}_{\theta^*}:v_1(\theta)=X(\theta,\theta^*)\}, \ \theta_2^*=\essinf\{\theta\in\mathcal{S}_{\theta^*}:v_2(\theta)=X(\theta^*,\theta)\}
		\end{displaymath}
		are optimal for the value function $v_1(\theta^*)$ and $v_2(\theta^*)$, respectively. Consider the following two stopping times
		\begin{displaymath}
		\tau_1^*=\theta^* I_B+\theta_1^* I_{B^c}, \ \tau_2^*=\theta_2^* I_B+\theta^*I_{B^c},
		\end{displaymath}
		where $B=\{v_1(\theta^*)\leq v_2(\theta^*)\}=\{u_1(\theta^*)\leq u_2(\theta^*)\}$ by the definition of $u_1$ and $u_2$ in \eqref{2.2}. By Proposition \ref{P6}, we derive that $(\tau^*_1,\tau_2^*)$ is optimal for $v(S)$.
	\end{proof}

Since $v$ defined by \eqref{2.1} coincides with the value function of the optimal single stopping problem with the reward family $\{\widetilde{X}(\tau),\tau\in\mathcal{S}_0\}$, by Proposition \ref{P3} and \ref{P4}, $\{v(\tau),\tau\in\mathcal{S}_0\}$ is C$\mathcal{E}$ if $\{\widetilde{X}(\tau),\tau\in\mathcal{S}_0\}$  is C$\mathcal{E}$.

\begin{corollary}
	Under the same hypothesis as Theorem \ref{T4}, the family $\{v(\tau),\tau\in\mathcal{S}_0\}$ defined by \eqref{2.1} is C$\mathcal{E}$.
\end{corollary}

\begin{remark}
	By Proposition \ref{P3}, the RC$\mathcal{E}$ property of $\{v(\tau),\tau\in\mathcal{S}_0\}$ does not depend on the existence of optimal stopping times. Thus, the conditions can be weaken as the one in Theorem \ref{T3} to guarantee the RC$\mathcal{E}$ property  of $\{v(\tau),\tau\in\mathcal{S}_0\}$.
\end{remark}

	\section{The optimal $d$-stopping time problem under nonlinear expectation}

	As in Section 3, we assume that the $\mathbb{F}$-expectation $(\mathcal{E},\textrm{Dom}(\mathcal{E}))$ satifies  Assumptions (H0)-(H5). Now we introduce the optimal $d$-stopping times problem. The reward family should satisfy the following conditions.
	\begin{definition}
		A family of random variables $\{X(\tau),\tau\in\mathcal{S}_0^d\}$ is said to be $d$-admissible if it satisfies the following conditions:
		\begin{description}
			\item[(1)] for all $\tau=(\tau_1,\cdots,\tau_d)\in\mathcal{S}_0^d$, $X(\tau)\in \textrm{Dom}^+_{\tau_1\vee\cdots\tau_d}(\mathcal{E})$;
			\item[(2)] for all $\tau,\sigma\in\mathcal{S}_0^d$, $X(\tau)=X(\sigma)$ a.s. on $\{\tau=\sigma\}$.
		\end{description}
	\end{definition}
	
	For each fixed stopping time $S\in\mathcal{S}_0$, the value function of the optimal $d$-stopping time problem associated with reward family $\{X(\tau),\tau\in\mathcal{S}_0^d\}$ is given by
	\begin{equation}\label{3.1}
	v(S)=\esssup_{\tau\in\mathcal{S}_S^d}\mathcal{E}_S[X(\tau)]=\esssup\{\mathcal{E}_S[X(\tau_1,\cdots,\tau_d)],\tau_1,\cdots,\tau_d\in\mathcal{S}_S\}.
	\end{equation}
	Similar with the optimal double stopping time case, the family $\{v(S),S\in\mathcal{S}_0\}$ is admissible and is an $\mathcal{E}$-supermartingale system as the following proposition shows.
	
	\begin{proposition}\label{P3.1}
		Let $\{X(\tau),\tau\in\mathcal{S}_0^d\}$ be a $d$-admissible family of random variables with $\sup_{\tau\in\mathcal{S}_0^d}\mathcal[X(\tau)]<\infty$. Then, the value function $\{v(S),S\in\mathcal{S}_0\}$ defined by \eqref{3.1} satisfies the following properties:
		\begin{description}
			\item[(i)] $\{v(S),S\in\mathcal{S}_0\}$ is an admissible family;
			\item[(ii)] for each $S\in\mathcal{S}_0$, there exists a sequence of  stopping times $\{\tau^n\}_{n\in\mathbb{N}}\subset \mathcal{S}_S^d$ such that $\mathcal{E}_S[X(\tau^n)]$ converges monotonically up to $v(S)$;
			\item[(iii)] $\{v(S),S\in\mathcal{S}_0\}$ is an $\mathcal{E}$-supermartingale system;
			\item[(iv)] for each $S\in\mathcal{S}_0$, we have $\mathcal{E}[v(S)]=\sup_{\tau\in\mathcal{S}_S^d}\mathcal{E}[X(\tau)]$.
		\end{description}
	\end{proposition}
	
	\begin{proof}
		The proof is similar with the one of Proposition \ref{P5}. We omit it.
	\end{proof}

	In the following, we will interpret the value function $v(S)$ defined in \eqref{3.1} as the value function of an optimal single stopping problem associated with a new reward family. For this purpose, for each $i=1,\cdots,d$ and $\theta\in\mathcal{S}_0$, consider the following random variable
	\begin{equation}\label{3.2}
	u^{(i)}(\theta)=\esssup_{\tau\in\mathcal{S}_\theta^{d-1}}\mathcal{E}_\theta[X^{(i)}(\tau,\theta)],
	\end{equation}
	where
	\begin{equation}\label{3.5}
	X^{(i)}(\tau_1,\cdots,\tau_{d-1},\theta)=X(\tau_1,\cdots,\tau_{i-1},\theta,\tau_{i+1},\cdots,\tau_{d-1}).
	\end{equation}
	It is easy to see that $u^{(i)}(\theta)$ is the value function of the optimal $(d-1)$-stopping problem corresponding to the reward $\{X^{(i)}(\tau,\theta),\tau\in\mathcal{S}_\theta^{d-1}\}$. Now we define
	\begin{equation}\label{3.3}
	\widehat{X}(\theta)=\max\{u^{(1)}(\theta),\cdots,u^{(d)}(\theta)\},
	\end{equation}
	and
	\begin{equation}\label{3.4}
	u(S)=\esssup_{\tau\in\mathcal{S}_S}\mathcal{E}_S[\widehat{X}(\tau)].
	\end{equation}
	The following theorem indicates that the value function $v$ defined by \eqref{3.1} coincides with $u$.
	
	\begin{theorem}\label{T5}
		Let $\{X(\tau),\tau\in\mathcal{S}_0^d\}$ be a $d$-admissible family with $\sup_{\tau\in\mathcal{S}_0^d}\mathcal{E}[X(\tau)]<\infty$. Then, for any $S\in\mathcal{S}_0$, we have $v(S)=u(S)$.
	\end{theorem}
	
	\begin{proof}
		By the definition of $v$ and $u^{(i)}$, it is obvious that $v(S)\geq u^{(i)}(S)$, for any $i=1,\cdots,d$ and $S\in\mathcal{S}_0$. Therefore, we have $v(S)\geq \widehat{X}(S)$, for any $S\in\mathcal{S}_0$. By Propositions \ref{P1} and \ref{P3.1}, $\{v(S),S\in\mathcal{S}_0\}$ is an $\mathcal{E}$-supermartingale system which dominates $\{\widehat{X}(S),S\in\mathcal{S}_0\}$ while $\{u(S),S\in\mathcal{S}_0\}$ is the smallest one which does so. It follows that $v(S)\geq u(S)$.
		
		It remains to show the reverse inequality. For each fixed $S\in\mathcal{S}_0$, consider the multiple stopping time $\tau=(\tau_1,\cdots,\tau_d)\in\mathcal{S}_S^d$. There exists a disjoint partition $\{A_i\}_{i=1}^d$ of $\Omega$ such that $\tau_1\wedge\cdots\wedge\tau_d=\tau_i$ on $A_i$ and $A_i$ belongs to $\mathcal{F}_{\tau_1\wedge\cdots\wedge\tau_d}$ for $i=1,\cdots,d$. It is easy to check that
		\begin{displaymath}
		\mathcal{E}_{\tau_i}[X(\tau)]I_{A_i}\leq u^{(i)}(\tau_i)I_{A_i}\leq \widehat{X}(\tau_i)I_{A_i}=\widehat{X}(\tau_1\wedge\cdots\wedge\tau_d)I_{A_i}.
		\end{displaymath}
		By simple calculation, we obtain that
		\begin{displaymath}
		\mathcal{E}_{S}[X(\tau)]=\mathcal{E}_S[\sum_{i=1}^d\mathcal{E}_{\tau_1\wedge\cdots\wedge\tau_d}[X(\tau)]I_{A_i}]=\mathcal{E}_S[\sum_{i=1}^d\mathcal{E}_{\tau_i}[X(\tau)]I_{A_i}]\leq \mathcal{E}_S[\widehat{X}(\tau_1\wedge\cdots\wedge\tau_d)]\leq u(S).
		\end{displaymath}
		Taking supremum over all $\tau\in\mathcal{S}_S^d$ yields that $v(S)\leq u(S)$. The proof is complete.
	\end{proof}
	
	With the above characterization of the value function, we may propose a possible construction of the optimal multiple stopping times by induction..
	\begin{proposition}\label{P3.3}
		For any fixed $S\in\mathcal{S}_0$, suppose that
		\begin{description}
			\item[1.] there exists $\theta^*\in\mathcal{S}_S$ such that $u(S)=\mathcal{E}_S[\widehat{X}(\theta^*)]$;
			\item[2.] for any $i=1,\cdots,d$, there exists $\theta^{(i)*}=(\theta_1^{(i)*},\cdots,\theta_{i-1}^{(i)*},\theta_{i+1}^{(i)*},\cdots,\theta_d^{(i)*})\in\mathcal{S}_{\theta^*}^{d-1}$ such that $u^{(i)}(\theta^*)=\mathcal{E}_{\theta^*}[X^{(i)}(\theta^{(i)*},\theta^*)]$.
		\end{description}
		Let $\{B_i\}_{i=1}^d$ be an $\mathcal{F}_{\theta^*}$-measurable and disjoint partition of $\Omega$ such that $\widehat{X}(\theta^*)=u^{(i)}(\theta^*)$ on the set $B_i$, $i=1,\cdots,d$. Set
		\begin{equation}\label{3.7}
		\tau_j^*=\theta^* I_{B_j}+\sum_{i\neq j,i=1}^d \theta^{(i)*}_j I_{B_i}.
		\end{equation}
		Then, $\tau^*=(\tau_1^*,\cdots,\tau_d^*)$ is optimal for $v(S)$, and $\tau_1^*\wedge\cdots\wedge\tau_d^*=\theta^*$.
	\end{proposition}
	
	\begin{proof}
		It is easy to check that $\tau^*\in\mathcal{S}_S^d$ and $\tau_1^*\wedge\cdots\wedge\tau_d^*=\theta^*$. By simple calculation, we obtain that
		\begin{align*}
		v(S)&=u(S)=\mathcal{E}_S[\widehat{X}(\theta^*)]=\mathcal{E}_S[\sum_{i=1}^d u^{(i)}(\theta^*)I_{B_i}]=\mathcal{E}_S[\sum_{i=1}^d \mathcal{E}_{\theta^*}[X^{(i)}(\theta^{(i)*},\theta^*)]I_{B_i}]\\
		&=\mathcal{E}_S[ \mathcal{E}_{\theta^*}[\sum_{i=1}^dX^{(i)}(\theta^{(i)*},\theta^*)I_{B_i}]]=\mathcal{E}_S[X(\tau^*)],
		\end{align*}
		which implies the optimality of $\tau^*$.
	\end{proof}

    \begin{proposition}
    	For any fixed $S\in\mathcal{S}_0$, suppose that $\tau^*=(\tau_1^*,\cdots,\tau_d^*)$ is optimal for $v(S)$. Then, we have
    	\begin{description}
    		\item[(1)] $\tau_1^*\wedge\cdots\wedge \tau_d^*$ is optimal for $u(S)$;
    		\item[(2)] for any $i=1,\cdots,d$, $(\tau_1^*,\cdots,\tau_{i-1}^*,\tau^*_{i+1},\cdots,\tau_d^*)$ is optimal for $u^{(i)}(\tau_i^*)$ on the set $\{\tau_1^*\wedge\cdots\wedge \tau_d^*=\tau_i^*\}$.
    	\end{description}
    \end{proposition}

    \begin{proof}
    	When we replace $\tau=(\tau_1,\cdots,\tau_d)$ by $\tau^*=(\tau_1^*,\cdots,\tau_d^*)$ in the proof of Theorem \ref{T5}, all the inequalities turn into equalities. The proof is complete.
    \end{proof}

    \begin{remark}
    	All the above results in this section do not need any regularity assumption on the reward family $\{X(\tau),\tau\in\mathcal{S}_0^d\}$.
    \end{remark}

    	The definition of continuity for the reward with $d$-parameters is similar with the one for the double stopping case.
    \begin{definition}
    	A $d$-admissible family $\{X(\tau),\tau\in\mathcal{S}_0^d\}$ is said to be right-continuous (resp. left-continuous) along stopping times in $\mathcal{E}$-expectation [RC$\mathcal{E}$ (resp., LC$\mathcal{E}$)] if, for any $\tau\in\mathcal{S}_0^d$ and any sequence $\{\tau_n\}_{n\in\mathbb{N}}\subset \mathcal{S}_0^d$ such that $\tau_n\downarrow \tau$ (resp., $\tau_n\uparrow \tau$), one has $\mathcal{E}[X(\tau)]=\lim_{n\rightarrow\infty}\mathcal{E}[X(\tau_n)]$. If the family $\{X(\tau),\tau\in\mathcal{S}_0^d\}$ is both RC$\mathcal{E}$ and LC$\mathcal{E}$, it is said to be continuous along stopping times in $\mathcal{E}$-expectation (C$\mathcal{E}$).
    \end{definition}

    \begin{proposition}\label{P3.2}
    	Let $\{X(\tau),\tau\in\mathcal{S}_0^d\}$ be an RC$\mathcal{E}$ $d$-admissible family with $\sup_{\tau\in\mathcal{S}_0^d}\mathcal{E}[X(\tau)]<\infty$. Then, the family $\{v(S),S\in\mathcal{S}_0\}$ is RC$\mathcal{E}$.
    \end{proposition}

    \begin{proof}
    	The proof is similar with the one of Proposition \ref{P3}. We omit it.
    \end{proof}

    \begin{remark}\label{r3}
    	Similar with the analysis of Remark \ref{r1}, suppose that $\{X(\tau),\tau\in\mathcal{S}_0^d\}$ is a $d$-admissible family with $\sup_{\tau\in\mathcal{S}_0^d}\mathcal{E}[X(\tau)]<\infty$ and right-continuous in $\mathcal{E}$-expectation along stopping times greater than $\sigma$ (i.e., if a sequence of stopping times $\{\tau_n\}_{n\in\mathbb{N}}\subset \mathcal{S}_\sigma^d$ satisfies $\tau_n\downarrow \tau$, then one has $\mathcal{E}[X(\tau)]=\lim_{n\rightarrow\infty}\mathcal{E}[X(\tau_n)]$). Then, the family of value functions $\{v(S),S\in\mathcal{S}_0\}$ is right-continuous in $\mathcal{E}$-expectation along stopping times greater than $\sigma$.
    \end{remark}

	By Theorem \ref{T5} and Proposition \ref{P3.3}, the value function and the optimal multiple stopping times of the optimal $d$-stopping problem can be constructed by the ones of the optimal $(d-1)$-stopping problem. Therefore, by induction, the multiple stopping problem can be reduced to nested single stopping problems. Besides, the existence of the optimal stopping time for the single stopping problem associated with the new reward $\{\widehat{X}(S),S\in\mathcal{S}_0\}$ is the building block for constructing the optimal stopping time for the original $d$-stopping problem. According to Theorem \ref{T1}, it remains to investigate the regularity of this new reward family.
	
	\begin{definition}
		A $d$-admissible family $\{X(\tau),\tau\in\mathcal{S}_0^d\}$ is said to be uniformly right-continuous (resp. left-continuous) along stopping times in $\mathcal{E}$-expectation [URC$\mathcal{E}$ (resp., ULC$\mathcal{E}$)] if for each $i=1,\cdots,d$, $S\in\mathcal{S}_0$ and a sequence of stopping times $\{S_n\}_{n\in\mathbb{N}}$ such that $S_n\downarrow S$ (resp., $S_n\uparrow S$), one has
		\begin{displaymath}
		\lim_{n\rightarrow\infty}\sup_{\theta\in\mathcal{S}_0^{d-1}}\mathcal{E}[|X^{(i)}(\theta,S_n)-X^{(i)}(\theta,S)|]=0.
		\end{displaymath}
	\end{definition}
	
	\begin{proposition}\label{P3.4}
		Let $(\widetilde{\mathcal{E}},\textrm{Dom}(\widetilde{\mathcal{E}}))$ be an $\mathbb{F}$-expectation satisfying Assumptions (H0)-(H5). Suppose that the $\mathbb{F}$-expectation $(\mathcal{E},\textrm{Dom}(\mathcal{E}))$ is  dominated by $(\widetilde{\mathcal{E}},\textrm{Dom}(\widetilde{\mathcal{E}}))$ and $\{X(\tau),\tau\in\mathcal{S}_0^d\}$ is a URC$\widetilde{\mathcal{E}}$ $d$-admissible family  with $\sup_{\tau\in\mathcal{S}_0^d}\mathcal{E}[X(\tau)]<\infty$. Then, the family $\{\widehat{X}(\tau),\tau\in\mathcal{S}_0\}$ defined by \eqref{3.3} is RC$\mathcal{E}$.
	\end{proposition}
	
	\begin{proof}
		The proof is similar with the one of Theorem \ref{T3}, so we omit it.
	\end{proof}

 Since the left-continuity along stopping times in $\mathcal{E}$-expectation relies on the existence of optimal stopping times, the conditions under which the LC$\mathcal{E}$ holds is more restrictive than the RC$\mathcal{E}$ case and the proof of LC$\mathcal{E}$ is more complicated as explained before Theorem \ref{t2} in Secion 3.
    	\begin{proposition}\label{P3.5}
    	Suppose that the $\mathbb{F}$-expectation $(\mathcal{E},\textrm{Dom}(\mathcal{E}))$ satisfies (H0)-(H7) and $\{X(\tau),\tau\in\mathcal{S}_0^d\}$ is a UC${\mathcal{E}}$ $d$-admissible family (i.e., both URC$\mathcal{E}$ and ULC$\mathcal{E}$) with $\sup_{\tau\in\mathcal{S}_0^d}\mathcal{E}[X(\tau)]<\infty$. Then, the family $\{\widehat{X}(\tau),\tau\in\mathcal{S}_0\}$ defined by \eqref{3.3} is LC$\mathcal{E}$.
    \end{proposition}

    \begin{proof}
    	By Proposition \ref{P6} and Theorem \ref{t2}, this result holds for the cases $d=1,2$. We only consider the case that $d=3$ and the other cases can be proved similarly.  By the definition of $\widehat{X}$, it is sufficient to prove that $\{u^{(i)}(S),S\in\mathcal{S}_0\}$ is LC$\mathcal{E}$. For any given $\theta\in\mathcal{S}_0$, let $\{\theta_n\}_{n\in\mathbb{N}}$ be a sequence of stopping times such that $\theta_n\uparrow\theta$. Set
    	\begin{displaymath}
    		X^{i}(\tau_1,\tau_2,\theta)=X^{(i)}(\tau_1,\tau_2,\theta)I_{\{\tau_1\vee\tau_2\geq \theta\}}-I_{\{\tau_1\vee\tau_2<\theta\}}, \ i=1,2,3,
    	\end{displaymath}
    	and
    	\begin{align*}
    		&X^{i,1}(\tau,\theta)=X^i(\theta,\tau,\theta)=X^{(i)}(\theta,\tau,\theta)I_{\{\tau\geq \theta\}}-I_{\{\tau<\theta\}},\\
    		&X^{i,2}(\tau,\theta)=X^i(\tau,\theta,\theta)=X^{(i)}(\tau,\theta,\theta)I_{\{\tau\geq \theta\}}-I_{\{\tau<\theta\}}.
    	\end{align*}
    	It is easy to check that the families $\{X^{i,j}(\tau,\theta),\tau\in\mathcal{S}_0\}$ are admissible and $\{X^i(\tau_1,\tau_2,\theta),\tau_1,\tau_2\in\mathcal{S}_0\}$ are biadmissible, where $i=1,2,3$ and $j=1,2$. We claim that the families $\{X^{i,j}(\tau,\theta),\tau\in\mathcal{S}_0\}$ are RC$\mathcal{E}$. In fact, consider a sequence of stopping times $\{\tau_n\}_{n\in\mathbb{N}}$ such that $\tau_n\downarrow\tau$. We deduce that
    	\begin{displaymath}
    		\mathcal{E}[|X^{i,1}(\tau_n,\theta)-X^{i,1}(\tau,\theta)|]\leq \mathcal{E}[|X^{(i)}(\theta,\tau_n,\theta)-X^{(i)}(\theta,\tau,\theta)|]+\mathcal{E}[\eta I_{\{\tau_n\geq \theta>\tau\}}],
    	\end{displaymath}
    	where $\eta=1+\esssup_{\tau\in\mathcal{S}_0}X^{(i)}(\theta,\tau,\theta)$. Applying Lemma \ref{L4.1} yields that $\eta\in\textrm{Dom}^+(\mathcal{E})$.  Noting that $I_{\{\tau_n\geq \theta>\tau\}}\downarrow 0$ and the family  $\{X(\tau_1,\tau_2,\tau_3),\tau_1,\tau_2,\tau_3\in\mathcal{S}\}$ is URC$\mathcal{E}$, we have
    	\begin{displaymath}
    		\lim_{n\rightarrow\infty}\mathcal{E}[|X^{i,1}(\tau_n,\theta)-X^{i,1}(\tau,\theta)|]=0.
    	\end{displaymath}
    	Hence, the claim follows.
    	 We now define the following value function
    	\begin{displaymath}
    		u^{i,j}(S)=\esssup_{\tau\in\mathcal{S}_S}\mathcal{E}_S[X^{i,j}(\tau,\theta)].
    	\end{displaymath}
    	By Proposition \ref{P3} and Theorem \ref{t1}, the family $\{u^{i,j}(S),S\in\mathcal{S}_0\}$ is C$\mathcal{E}$ and the optimal stopping time is greater than $\theta$. Then the family $\{\hat{u}^i(S),S\in\mathcal{S}_0\}$ inherits the properties of $\{u^{i,j}(S),S\in\mathcal{S}_0\}$, where $\hat{u}^i(S)=\max\{u^{i,1}(S),u^{i,2}(S)\}$. By Theorem \ref{T5}, the value function of the optimal single stopping problem with reward family $\{\hat{u}^i(S),S\in\mathcal{S}_0\}$, denoted by $u^i$,  coincides with the one of the optimal double stopping problem with reward family $\{X^i(\tau_1,\tau_2,\theta),\tau_1,\tau_2\in\mathcal{S}_0\}$, denoted by $u^{(i),\theta}$, that is
    	\begin{displaymath}
    		\esssup_{\tau\in\mathcal{S}_S}\mathcal{E}_S[\hat{u}^i(\tau)]={u}^i(S)=u^{(i),\theta}(S)=\esssup_{\tau_1,\tau_2\in\mathcal{S}_S}\mathcal{E}_S[X^i(\tau_1,\tau_2,\theta)].
    	\end{displaymath}
    	Applying Proposition \ref{P4}, the family $\{u^i(S),S\in\mathcal{S}_0\}$ is LC$\mathcal{E}$, which implies the LC$\mathcal{E}$ property of $\{u^{(i),\theta}(S),S\in\mathcal{S}_0\}$. Recalling the definition of $X^i$, we obtain that
    	\begin{displaymath}
    		\mathcal{E}[u^{(i)}(\theta)]=\mathcal{E}[\esssup_{\tau_1,\tau_2\in\mathcal{S}_\theta}\mathcal{E}_\theta[X^{(i)}(\tau_1,\tau_2,\theta)]]=\mathcal{E}[\esssup_{\tau_1,\tau_2\in\mathcal{S}_\theta}\mathcal{E}_\theta[X^i(\tau_1,\tau_2,\theta)]]=\lim_{n\rightarrow\infty}\mathcal{E}[u^{(i),\theta}(\theta_n)].
    	\end{displaymath}
    	By a similar analysis as the proof of Theorem \ref{t2}, we have
    	\begin{displaymath}
    		|\mathcal{E}[u^{(i),\theta}(\theta_n)]-\mathcal{E}[u^{(i)}(\theta_n)]|\leq \sup_{\tau_1,\tau_2\in\mathcal{S}_0}\mathcal{E}[|X^{(i)}(\tau_1,\tau_2,\theta)-X^{(i)}(\tau_1,\tau_2,\theta_n)|]+\mathcal{E}[\xi I_{\{\theta_n<\theta\}}],
    	\end{displaymath}
    	where $\xi=1+\esssup_{\tau=(\tau_1,\tau_2,\tau_3)\in\mathcal{S}_0^3}X(\tau)$. By Assumption (H2) and the ULC$\mathcal{E}$ property, we deduce that
    	\begin{displaymath}
    		\lim_{n\rightarrow\infty}|\mathcal{E}[u^{(i),\theta}(\theta_n)]-\mathcal{E}[u^{(i)}(\theta_n)]|=0.
    	\end{displaymath}
    	Hence, $\lim_{n\rightarrow\infty}\mathcal{E}[u^{(i)}(\theta_n)]=\mathcal{E}[u^{(i)}(\theta)]$, which completes the proof.
    \end{proof}
	
	With the help of Propositions \ref{P3.3}, \ref{P3.4} and \ref{P3.5}, we can now establish the existence result of the optimal stopping times for the multiple stopping problem.
	\begin{theorem}\label{T6}
		 Suppose that the $\mathbb{F}$-expectation $(\mathcal{E},\textrm{Dom}(\mathcal{E}))$ satisfies all the Assumptions (H0)-(H7)  and $\{X(\tau),\tau\in\mathcal{S}_0^d\}$ is a UC${\mathcal{E}}$ $d$-admissible family with $\sup_{\tau\in\mathcal{S}_0^d}\mathcal{E}[X(\tau)]<\infty$. Then, there exists an optimal stopping time $\tau^*\in\mathcal{S}_S^d$ for $v(S)$, that is
		\begin{displaymath}
		v(S)=\esssup_{\tau\in\mathcal{S}_S^d}\mathcal{E}_S[X(\tau)]=\mathcal{E}_S[X(\tau^*)].
		\end{displaymath}
	\end{theorem}
	
	\begin{proof}
		We prove this result by induction. Indeed, the result holds true for the case $d=1,2$ by Theorems \ref{T1} and \ref{T4}. Fixed $d\geq 1$, suppose that the optimal stopping problem exists for all value functions induced by UC${\mathcal{E}}$ $d$-admissible families. Let $\{X(\tau),\tau\in\mathcal{S}_0^{d+1}\}$ be a $(d+1)$-admissible family which is UC${\mathcal{E}}$. By Proposition \ref{P3.4} and \ref{P3.5}, the corresponding reward family $\{\widehat{X}(\tau),\tau\in\mathcal{S}_0\}$ obtained by \eqref{3.2} and \eqref{3.3} is C$\mathcal{E}$. Hence, Theorem \ref{T1} shows that there exists an optimal stopping time $\theta^*$ for $u(S)$ defined by \eqref{3.4}. It is easy to check that the $d$-admissible family $\{X^{(i)}(\theta,\theta^*),\theta\in\mathcal{S}_{\theta^*}^d\}$ is UC${\mathcal{E}}$ for any $i=1,\cdots,d+1$, where $X^{(i)}$ is given by \eqref{3.5}. Therefore, by the induction assumption, there exists an optimal $\theta^{(i)*}\in\mathcal{S}_{\theta^*}^d$ for the value function $u^{(i)}(\theta^*)$. By Proposition \ref{P3.3}, we may construct the optimal stopping time $\tau^*\in\mathcal{S}_0^{d+1}$ for the value function corresponding to the $(d+1)$-admissible family $\{X(\tau),\tau\in\mathcal{S}_0^{d+1}\}$. The proof is complete.
	\end{proof}
	
	In order to characterize the optimal multiple stopping times in a minimal way, we should first define a partial order relation $\prec_d$ on $\mathbb{R}^d$. This relation can be found in \cite{KQR} and for readers' convenience, we list it here: for $d=1$ and any $a,b\in\mathbb{R}$, $a\prec_1 b$ if and only if $a\leq b$, and for $d>1$ and any $(a_1,\cdots,a_d),(b_1,\cdots,b_d)\in\mathbb{R}^d$, $(a_1,\cdots,a_d)\prec_d(b_1,\cdots,b_d)$ if and only if either $a_1\wedge\cdots\wedge a_d<b_1\wedge \cdots\wedge b_d$ or  
	\begin{displaymath}
		\begin{cases}
		a_1\wedge\cdots\wedge a_d=b_1\wedge \cdots\wedge b_d, \textrm{ and, for } i=1,2,\cdots,d,\\
		a_i=a_1\wedge\cdots\wedge a_d\Rightarrow\begin{cases}
		b_i=b_1\wedge \cdots\wedge b_d \textrm{ and }\\
		(a_1,\cdots,a_{i-1},a_{i+1}\cdots,a_d)\prec_{d-1}(b_1,\cdots,b_{i-1},b_{i+1}\cdots,b_d).
		\end{cases}
		\end{cases}
	\end{displaymath}
	
	\begin{definition}
		For each fixed $S\in\mathcal{S}_0$, a $d$-stopping time $(\tau_1,\cdots,\tau_d)\in\mathcal{S}_S^d$ is said to be $d$-minimal optimal for the value function $v(S)$ defined by \eqref{3.1} if it is minimal for the order $\prec_d$ in the set $\{\tau\in\mathcal{S}_S^d:v(S)=\mathcal{E}_S[X(\tau)]\}$ which is the collection of all optimal stopping times.
	\end{definition}

    \begin{proposition}
    	For each fixed $S\in\mathcal{S}_0$, a $d$-stopping time $(\tau_1,\cdots,\tau_d)\in\mathcal{S}_S^d$ is  $d$-minimal optimal for the value function $v(S)$ defined by \eqref{3.1} if and only if:
    	\begin{description}
    		\item[(1)] $\theta^*=\tau_1\wedge\cdots\wedge \tau_d$ is the minimal optimal stopping time for $u(S)$ defined by \eqref{3.4};
    		\item[(2)] for $i=1,\cdots,d$, $\theta^{*(i)}=\tau_i\in\mathcal{S}_S^{d-1}$ is the $(d-1)$-minimal optimal stopping time for $u^{(i)}(\theta^*)$ defined by \eqref{3.2} on the set $\{u^{(i)}(\theta^*)\geq \vee_{k\neq i} u^{(k)}(\theta^*)\}$.
    	\end{description}
    \end{proposition}

	\section{Aggregation of the optimal multiple stopping problem}

	We first recall some basic results in \cite{CR}. Suppose that the $\mathbb{F}$-expectation $(\mathcal{E},\textrm{Dom}(\mathcal{E}))$ is reduced to the $g$-expectation $(\{\mathcal{E}^g_t[\cdot]\}_{t\in[0,T]}, L^2(\mathcal{F}_T))$ satisfying the assumptions in Example 1.1. Now, given an adapted, nonnegative process $\{X_t\}_{t\in[0,T]}$ which has continuous sample path with $E[\sup_{t\in[0,T]}X_t^2]<\infty$, the value function is defined by:
	\begin{displaymath}
		v_t^g=\esssup_{\tau\in\mathcal{S}_t}\mathcal{E}_t^g[X_\tau].
	\end{displaymath}
	Cheng and Riedel \cite{CR} proves that the first hitting time
	\begin{displaymath}
		\tau^*=\inf\{t\geq 0: v_t^g=X_t \}
	\end{displaymath}
	is an optimal stopping time. This formulation makes it efficient to compute an optimal stopping time.
	
	In this section, we aim to express the optimal stopping times studied in the previous parts by the hitting times of processes. According to Theorem \ref{T6}, the multiple optimal stopping times can be constructed  by the induction method.  Therefore, it is sufficient to study the double stopping case, which remains to aggregate the value function and the reward family. For this purpose, we need to make some stronger regularity conditions.
	
	In the following part of this section, assume that the $\mathbb{F}$-expectation $(\mathcal{E},\textrm{Dom}(\mathcal{E}))$ satisfies (H0)-(H5).	The following proposition can be used to aggregate the value function of both the single and multiple stopping problem.
	\begin{proposition}\label{P4.1}
		Let $\{h(\tau),\tau\in\mathcal{S}_0\}$ be a nonnegative, RC$\mathcal{E}$ $\mathcal{E}$-supermaringale system with $h(0)<\infty$. Then, there exists an adapted process $\{h_t\}_{t\in[0,T]}$ which is RCLL such that it aggregates the family $\{h(\tau),\tau\in\mathcal{S}_0\}$, i.e., $h_\tau=h(\tau)$, for any $\tau\in\mathcal{S}_0$.
	\end{proposition}
	
	\begin{proof}
		Consider the process $\{h(t)\}_{t\in[0,T]}$. Since this process is an $\mathcal{E}$-supermartingale and the function $t\rightarrow\mathcal{E}[h(t)]$ is right-continuous, by Proposition \ref{P2.2}, there is an $\mathcal{E}$-supermartingale $\{h_t\}_{t\in[0,T]}$ which is RCLL such that for each $t\in[0,T]$, $h_t=h(t)$, a.s. For each $n\in\mathbb{N}$, set $\mathcal{I}_n=\{0,\frac{1}{2^n}\wedge T, \frac{2}{2^n}\wedge T,\cdots, T\}$ and $\mathcal{I}=\cup_{n=1}^\infty\mathcal{I}_n$. Then, for any stopping time $\tau$ taking values in $\mathcal{I}$, we have $h_\tau=h(\tau)$, a.s., which implies that
		\begin{equation}\label{4.7}
		\mathcal{E}[h(\tau)]=\mathcal{E}[h_\tau].
		\end{equation}
		For any stopping time $\tau\in\mathcal{S}_0$, we may construct a sequence of stopping times $\{\tau_n\}_{n\in\mathbb{N}}$ which takes values in $\mathcal{I}$, such that $\tau_n\downarrow\tau$. Noting that $\{h_t\}_{t\in[0,T]}$ is RCLL, then $h_{\tau_n}$ converges to $h_\tau$. It is obvious that $h_{\tau_n}\leq \esssup_{\tau\in\mathcal{S}_0}h(\tau)=:\eta$. Since $\{h(\tau),\tau\in\mathcal{S}_0\}$ is an $\mathcal{E}$-supermartingale system, we have
		\begin{displaymath}
			\sup_{\tau\in\mathcal{S}_0}\mathcal{E}[h(\tau)]\leq h(0)<\infty.
		\end{displaymath}
		Then by Lemma \ref{L4.1}, we obtain that $\eta\in\textrm{Dom}^+(\mathcal{E})$. Noting that  $\{h(\tau),\tau\in\mathcal{S}_0\}$ is RC$\mathcal{E}$ and applying the dominated convergence theorem \ref{P2.8}, we may check that
		\begin{equation}\label{4.8}
		\mathcal{E}[h(\tau)]=\lim_{n\rightarrow\infty}\mathcal{E}[h(\tau_n)]=\lim_{n\rightarrow\infty}\mathcal{E}[h_{\tau_n}]=\mathcal{E}[h_\tau].
		\end{equation}
		Assume that $P(h_\tau\neq h(\tau))>0$. Without loss of generality, we may assume that $P(A)>0$, where $A=\{h_\tau>h(\tau)\}$. Set $\tau_A=\tau I_A+TI_{A^c}$. It is easy to check that $\tau_A$ is a stopping time and $h_{\tau_A}\geq h(\tau_A)$ with $P(h_{\tau_A}>h(\tau_A))=P(A)>0$. It follows that $\mathcal{E}[h(\tau_A)]<\mathcal{E}[h_{\tau_A}]$, which contradicts Equation \eqref{4.8}. Therefore, we obtain that $h_\tau=h(\tau)$ for any $\tau\in\mathcal{S}_0$.
	\end{proof}

    With the help of  Proposition \ref{P4.1}, the value function $\{v(\tau),\tau\in\mathcal{S}_0\}$ can be aggregated as an RCLL $\mathcal{E}$-supermartingale.

    \begin{proposition}\label{P4.2}
    	Let $\{X(\tau),\tau\in\mathcal{S}_0\}$ be an RC$\mathcal{E}$ admissible family with $\sup_{\tau\in\mathcal{S}_0}\mathcal{E}[X(\tau)]<\infty$. Then, there exists an RCLL $\mathcal{E}$-supermartingale $\{v_t\}_{t\in[0,T]}$ which aggregates the family $\{v(S),S\in\mathcal{S}_0\}$ defined in \eqref{1.1}, i.e., for each stopping time $S$, $v(S)=v_S$, a.s.
    \end{proposition}

    \begin{proof}
    	By Proposition \ref{P1} and \ref{P3}, $\{v(S),S\in\mathcal{S}_0\}$ is a nonnegative, RC$\mathcal{E}$ $\mathcal{E}$-supermartingale system. Recalling \eqref{1.3}, we have
    	\begin{displaymath}
    		v(0)=\mathcal{E}[v(0)]=\sup_{\tau\in\mathcal{S}_0}\mathcal{E}[X(\tau)]<\infty.
    	\end{displaymath}
    	The results follows from Proposition \ref{P4.1}.
    \end{proof}

    For the reward family $\{X(\tau),\tau\in\mathcal{S}_0\}$, since it is not an $\mathcal{E}$-supermartingale system, we cannot apply Proposition \ref{P4.1} to conclude that it can be aggregated. In order to do this, we need to require the following continuity property of the reward family.
	
	\begin{definition}[\cite{KQR}]
		An admissible family $\{X(\tau),\tau\in\mathcal{S}_0\}$ is said to be right-continuous along stopping times (RC) if for any $\tau\in\mathcal{S}_0$ and any sequence $\{\tau_n\}_{n\in\mathbb{N}}\subset\mathcal{S}_0$ such that $\tau_n\downarrow\tau$, one has $X(\tau)=\lim_{n\rightarrow\infty}X(\tau_n)$.
	\end{definition}

    \begin{remark}\label{R4}
    	If the admissible family $\{X(\tau),\tau\in\mathcal{S}_0\}$ is RC with $\sup_{\tau\in\mathcal{S}_0}\mathcal{E}[X(\tau)]<\infty$, then it is RC$\mathcal{E}$. Indeed, Let $\{\tau_n\}_{n\in\mathbb{N}}\subset\mathcal{S}_0$ be a sequence of stopping times such that $\tau_n\downarrow\tau$, a.s. By Lemma \ref{L4.1}, the random variable $\eta:=\esssup_{\tau\in\mathcal{S}_0}X(\tau)$ belongs to $\textrm{Dom}^+(\mathcal{E})$. Since $X(\tau_n)\leq \eta$, applying the dominated convergence theorem \ref{P2.8} implies that
    	\begin{displaymath}
    		\mathcal{E}[X(\tau)]=\lim_{n\rightarrow\infty}\mathcal{E}[X(\tau_n)].
    	\end{displaymath}
    \end{remark}
	
	The following theorem obtained in \cite{KQR} is used to aggregate the reward family.
	\begin{theorem}\label{T4.1}[\cite{KQR}]
		Suppose that the admissible family $\{X(\tau),\tau\in\mathcal{S}_0\}$ is right-continuous along stopping times. Then, there exists a progressively process $\{X_t\}_{t\in[0,T]}$ such that for each $\tau\in\mathcal{S}_0$, $X(\tau)=X_\tau$, a.s. and such that there exists a nonincreasing sequence of right-continuous processes $\{X^n_t\}_{t\in[0,T]}$ such that for each $(t,\omega)\in[0,T]\times\Omega$, $\lim_{n\rightarrow\infty}X_t^n(\omega)=X_t(\omega)$.
	\end{theorem}

    Now, we could prove that, the optimal stopping time for the single stopping problem obtained in Section 2 can be represented as the first hitting time.
    \begin{theorem}\label{T4.2}
    	Suppose that the $\mathbb{F}$-expectation satisfies all the Assumptions (H0)-(H7). Let $\{X(\tau),\tau\in\mathcal{S}_0\}$ be an RC and LC$\mathcal{E}$ admissible family with $\sup_{\tau\in\mathcal{S}_0}\mathcal{E}[X(\tau)]<\infty$. Then for any $S\in\mathcal{S}_0$, the optimal stopping time of $v(S)$ defined by \eqref{1.5} can be given by a first hitting times. More precisely, let $\{X_t\}_{t\in[0,T]}$ be the progressive process given by Theorem \ref{T4.1} that aggregates $\{X(\tau),\tau\in\mathcal{S}_0\}$ and let $\{v_t\}_{t\in[0,T]}$ be the RCLL $\mathcal{E}$-supermartingale that aggregates the family $\{v(\tau),\tau\in\mathcal{S}_0\}$. Then the random variable defined by
    	\begin{equation}\label{4.2}
    		{\tau}(S)=\inf\{t\geq S:v_t=X_t\}
    	\end{equation}
    	is the minimal optimal stopping time for $v(S)$.
    \end{theorem}

    \begin{proof}
    	For $\lambda\in(0,1)$, set
    	\begin{equation}\label{4.9}
    		\bar{\tau}^\lambda(S):=\inf\{t\geq S:\lambda v_t\leq X_t\}\wedge T.
    	\end{equation}
    	It is easy to check that the mapping $\lambda\mapsto \bar{\tau}^\lambda(S)$ is nondecreasing. Then the stopping time
    	\begin{displaymath}
    		\bar{\tau}(S)=\lim_{\lambda\uparrow 1}\bar{\tau}^\lambda(S)
    	\end{displaymath}
    	is well defined. The proof remains almost the same with the proofs of Lemma \ref{L1}, Lemma \ref{L2} and Theorem \ref{T1} if $\tau^\lambda(S)$, $\hat{\tau}(S)$ and $\tau^*(S)$ are replaced by $\bar{\tau}^\lambda(S)$, $\bar{\tau}(S)$ and $\tau(S)$ respectively except the proof for Equation \eqref{1.10}. In order to prove \eqref{1.10}, in the present setting, that is to prove the following inequality:
    	\begin{displaymath}
    		\lambda \mathcal{E}[v(\bar{\tau}^\lambda(S))]\leq \mathcal{E}[X(\bar{\tau}^\lambda(S))],
    	\end{displaymath}
    	it is suffient to verify that for each $S\in\mathcal{S}_0$ and $\lambda\in(0,1)$,
    	\begin{displaymath}
    		\lambda v_{\bar{\tau}^\lambda(S)}\leq X_{\bar{\tau}^\lambda(S)} ,\textrm{ a.s.}
    	\end{displaymath}
    	For the proof of this assertion, we may refer to Lemma 4.1 in \cite{KQR}. The proof is complete.
    \end{proof}

    In the following, we will show that the optimal stopping times for the multiple stopping problem can be given in terms of hitting times. For simplicity, we only consider the double stopping time problems. Let us first aggregate the value function.

    \begin{proposition}\label{P4.3}
    	Let $\{X(\tau,\sigma),\tau,\sigma\in\mathcal{S}_0\}$ be an RC$\mathcal{E}$ biadmissible family with $\sup_{\tau,\sigma\in\mathcal{S}_0}\mathcal{E}[X(\tau,\sigma)]<\infty$. Then, there exists an $\mathcal{E}$-supermartingale $\{v_t\}_{t\in[0,T]}$ with RCLL sample paths that aggregates the family $\{v(S),S\in\mathcal{S}_0\}$ defined by \eqref{2.1}, i.e., for each $S\in\mathcal{S}_0$, $v_S=v(S)$, a.s.
    \end{proposition}

    \begin{proof}
    	By Proposition \ref{P5} and \ref{P7}, the family $\{v(S),S\in\mathcal{S}_0\}$ is an $\mathcal{E}$-supermartingale system which is RC$\mathcal{E}$. Remark \ref{R2} implies that $v(0)<\infty$. Therefore, the result follows from Proposition \ref{P4.1}.
    \end{proof}

In order to aggregate the reward family obtained by \eqref{2.3}, by Theorem \ref{T4.1}, it suffices to show that it is RC. Since this new reward is defined by the value function of the single stopping problem corresponding to the biadmissible family, we need to assume that following regularity condition on the biadmissible family.
    \begin{definition}[\cite{KQR}]
    	A biadmissible family $\{X(\tau,\sigma),\tau,\sigma\in\mathcal{S}_0\}$ is said to be uniformly right-continuous along stopping times (URC) if $\sup_{\tau,\sigma\in\mathcal{S}_0}\mathcal{E}[X(\tau,\sigma)]<\infty$ and if for each nonincreasing sequence of stopping times $\{S_n\}_{n\in\mathbb{N}}\subset \mathcal{S}_S$ which converges a.s. to a stopping time $S\in\mathcal{S}_0$, one has
    	\begin{align*}
    		&\lim_{n\rightarrow\infty}[\esssup_{\tau\in\mathcal{S}_0}|X(\tau,S_n)-X(\tau,S)|]=0,\\
    		&\lim_{n\rightarrow\infty}[\esssup_{\sigma\in\mathcal{S}_0}|X(S_n,\sigma)-X(S,\sigma)|]=0.
    	\end{align*}
    \end{definition}

    \begin{theorem}\label{T4.3}
    	Suppose that there exists an $\mathbb{F}$-expectation $(\widetilde{\mathcal{E}},\textrm{Dom}(\widetilde{\mathcal{E}}))$ satisfying (H0)-(H5) that dominates $(\mathcal{E},\textrm{Dom}(\mathcal{E}))$. Let $\{X(\tau,\sigma),\tau,\sigma\in\mathcal{S}_0\}$ be a biadmissible family which is URC. Then, the family $\{\widetilde{X}(S),S\in\mathcal{S}_0\}$ defined by \eqref{2.3} is RC.
    \end{theorem}

    \begin{proof}
    	By the expression of $\widetilde{X}$, it is sufficient to prove that the family $\{u_1(\tau),\tau\in\mathcal{S}_0\}$ is RC. For any $\tau,\sigma\in\mathcal{S}_0$, we define
    	\begin{equation}\label{4.3}
    		U_1(\tau,\sigma)=\esssup_{\tau_1\in\mathcal{S}_\tau}\mathcal{E}_\tau[X(\tau_1,\sigma)].
    	\end{equation}
    	Since $u_1(\tau)=U_1(\tau,\tau)$, it remains to prove that $\{U_1(\tau,\sigma),\tau,\sigma\in\mathcal{S}_0\}$ is RC.
    	
    	Now let $\{\tau_n\}_{n\in\mathbb{N}},\{\sigma_n\}_{n\in\mathbb{N}}$ be two nonincreasing sequence of stopping times that converges to $\tau$ and $\sigma$ respectively. It is easy to check that
    	\begin{equation}\label{4.0}
    		|U_1(\tau,\sigma)-U_1(\tau_n,\sigma_n)|\leq |U_1(\tau,\sigma)-U_1(\tau_n,\sigma_n)|+|U_1(\tau_n,\sigma)-U_1(\tau_n,\sigma_n)|.
    	\end{equation}
    	
    	It is obvious that for each fixed $\sigma\in\mathcal{S}_0$, the family $\{X(\tau,\sigma),\tau\in\mathcal{S}_0\}$ is RC. By Remark \ref{R4}, this family is also RC$\mathcal{E}$. Note that $\{U_1(\tau,\sigma),\tau\in\mathcal{S}_0\}$ can be regarded as the value function of the single optimal stopping problem associated with the reward $\{X(\tau,\sigma),\tau\in\mathcal{S}_0\}$. Although the reward family $\{X(\tau,\sigma),\tau\in\mathcal{S}_0\}$ may not be admissible due to the lack of adaptedness, i.e., $X(\tau,\sigma)$ is not $\mathcal{F}_\tau$-measurable if $\tau<\sigma$, Remarks \ref{r2} and \ref{r1} imply that  $\{U_1(\tau,\sigma),\tau\in\mathcal{S}_0\}$ is an $\mathcal{E}$-supermartingale which is RC$\mathcal{E}$. By Proposition \ref{P4.2}, we obtain that there exists an RCLL adapted process $\{U_t^{1,\sigma}\}_{t\in[0,T]}$ such that for each stopping time $\tau\in\mathcal{S}_0$,
    	\begin{equation}\label{4.4}
    		U^{1,\sigma}_\tau=U_1(\tau,\sigma).
    	\end{equation}
    	Hence, the first part of the right-hand side of \eqref{4.0} can be written as $|U^{1,\sigma}_\tau-U^{1,\sigma}_{\tau_n}|$. Due to the right-continuity of $\{U_t^{1,\sigma}\}_{t\in[0,T]}$, it converges to $0$ as $n$ goes to infinity.
    	
    	For any $m\in\mathbb{N}$, set $Z_m=\sup_{r\geq m}\{\esssup_{\tau\in\mathcal{S}_0}|X(\tau,\sigma)-X(\tau,\sigma_r)|\}$. It is easy to check that
    	\begin{displaymath}
    		0\leq Z_m\leq 2\esssup_{\tau,\sigma\in\mathcal{S}_0} X(\tau,\sigma)=:\eta.
    	\end{displaymath}
    	A similar analysis as the proof of Lemma \ref{L4.1} shows that $\eta\in\textrm{Dom}^+(\mathcal{E})$. Therefore, $Z_m\in\textrm{Dom}^+(\mathcal{E})$ for any $m\in\mathbb{N}$. By simple calculation, for any $n\geq m$, we have
    	\begin{align*}
    		|U_1(\tau_n,\sigma)-U_1(\tau_n,\sigma_n)|\leq &\esssup_{\tau_1\in\mathcal{S}_{\tau_n}}|\mathcal{E}_{\tau_n}[X(\tau_1,\sigma)]-\mathcal{E}_{\tau_n}[X(\tau_1,\sigma_n)]|\\
    			\leq &\esssup_{\tau_1\in\mathcal{S}_{\tau_n}}\widetilde{\mathcal{E}}_{\tau_n}[|X(\tau_1,\sigma)-X(\tau_1,\sigma_n)|]\\
    				\leq &\widetilde{\mathcal{E}}_{\tau_n}[Z_m].
    	\end{align*}
    	Since for any $\xi\in \textrm{Dom}^+(\mathcal{E})$, the family $\{\widetilde{\mathcal{E}}_t[\xi]\}_{t\in[0,T]}$ is right-continuous, it follows that for any $m\in\mathbb{N}$,
    	\begin{equation}\label{4.5}
    		\limsup_{n\rightarrow\infty}|U_1(\tau_n,\sigma)-U_1(\tau_n,\sigma_n)|\leq \widetilde{\mathcal{E}}_{\tau}[Z_m].
    	\end{equation}
    	Note that $Z_m$ converges to $0$ as $m$ goes to infinity. By the dominated convergence theorem \ref{P2.8}, letting $m$ go to infinity in \eqref{4.5}, we obtain that the second term of the right-hand side of \eqref{4.0} converges to $0$. The proof is complete.
    \end{proof}

    Combining Theorems \ref{T4.1} and \ref{T4.3}, we may get the following aggregation result.
    \begin{corollary}\label{C4.1}
    	Under the same hypothesis as Theorem \ref{T4.3}, there exists some progressive right-continuous adapted process $\{\widetilde{X}_t\}_{t\in[0,T]}$ which aggregates the family $\{\widetilde{X}(\tau),\tau\in\mathcal{S}_0\}$, i.e., for any $\tau\in\mathcal{S}_0$, $\widetilde{X}_\tau=\widetilde{X}(\tau)$ a.s. and such that there exists a nonincreasing sequence of right-continuous processes $\{\widetilde{X}_t^n\}_{t\in[0,T]}$ that converges to $\{\widetilde{X}_t\}_{t\in[0,T]}$.
    \end{corollary}

    \begin{theorem}\label{T4.4}
    	 Suppose that the $\mathbb{F}$-expectation $(\mathcal{E},\textrm{Dom}(\mathcal{E}))$ satisfies all Assumptions (H0)-(H7) and the biadmissible family  $\{X(\tau,\sigma),\tau,\sigma\in\mathcal{S}_0\}$ is URC and ULC${\mathcal{E}}$. Then, the optimal stopping time for the value function defined by \eqref{2.1} can be given in term of some first hitting times.
    \end{theorem}

    \begin{proof}
    	Let $\{\widetilde{X}(\tau),\tau\in\mathcal{S}_0\}$ be the new reward family given by \eqref{2.3}. By Theorem \ref{t2} and Theorem \ref{T4.3}, it is LC$\mathcal{E}$ and RC. Applying Theorem \ref{T4.1}, there exists a progressively measurable process $\{\widetilde{X}_t\}_{t\in[0,T]}$ which aggregates this family. Let $\{u_t\}_{t\in[0,T]}$ be an RCLL process that aggregates the value function defined as \eqref{2.4} which corresponds to the reward family   $\{\widetilde{X}(\tau),\tau\in\mathcal{S}_0\}$ by Proposition \ref{P4.2}. Then Theorem \ref{T4.2} implies that, for any $S\in\mathcal{S}_0$, the stopping time
    	\begin{displaymath}
    		\theta^*=\inf\{t\geq S:u_t=\widetilde{X}_t\}
    	\end{displaymath}
    	is optimal for $u(S)$.
    	
    	For each $\theta\in\mathcal{S}_{\theta^*}$, set $X^{(1)}(\theta)=X(\theta,\theta^*)$ and $X^{(2)}(\theta)=X(\theta^*,\theta)$. For $i=1,2$, it is obvious that the family $\{X^{(i)}(\theta),\theta\in\mathcal{S}_{\theta^*}\}$ is admissible, RC and LC${\mathcal{E}}$. In order to aggregate this family using Theorem \ref{T4.1}, we need to extend its defintion to all stopping times $\theta\in\mathcal{S}_0$. One of the candidates is
    	\begin{displaymath}
    		\widetilde{X}^{(i)}(\theta)=X^{(i)}(\theta)I_{\{\theta\geq \theta^*\}}-I_{\{\theta<\theta^*\}}.
    	\end{displaymath}
    	It is easy to check that the family $\{\widetilde{X}^{(i)}(\theta),\theta\in\mathcal{S}_0\}$ is admissible, RC and left-continuous in expectation along stopping times greater than $\theta^*$. By Theorem \ref{T4.1}, there exists a progressive process $\{\widetilde{X}^{(i)}_t\}_{t\in[0,T]}$ that aggregates $\{\widetilde{X}^{(i)}(\theta),\theta\in\mathcal{S}_{0}\}$. Consider the following value function
    	\begin{displaymath}
    		\widetilde{v}^{(i)}(S)=\esssup_{\tau\in\mathcal{S}_S}\mathcal{E}_S[\widetilde{X}^{(i)}(\tau)].
    	\end{displaymath}
    	Applying Theorem \ref{t1}, we obtain that the family $\{\widetilde{v}^{(i)}(S),S\in\mathcal{S}_0\}$ is an RC$\mathcal{E}$ $\mathcal{E}$-supermartingale system. Furthermore, for any $S\geq \theta^*$, we have $\widetilde{v}^{(i)}(S)=u_i(S)$, where $u_i$ is defined by \eqref{2.2}. By Proposition \ref{P4.2}, there exists an RCLL process $\{\widetilde{v}^i_t\}_{t\in[0,T]}$ that aggregates the  family $\{\widetilde{v}^{(i)}(S),S\in\mathcal{S}_{0}\}$. Now, we define
    	\begin{displaymath}
    		\theta^*_i=\inf\{t\geq \theta^*:\widetilde{v}^i_t=\widetilde{X}^{(i)}_t\}.
    	\end{displaymath}
    	By a similar analysis as the proof of Theorem \ref{t1}, Theorem \ref{T4.2} still holds for the reward family given by $\{\widetilde{X}^{(i)}(\theta),\theta\in\mathcal{S}_0\}$, which implies that the stopping time $\theta^*_i$ is optimal for $\widetilde{v}_i(\theta^*)$, and then optimal for $u_i(\theta^*)$. Now, set $B=\{u_1(\theta^*)\leq u_2(\theta^*)\}=\{\widetilde{v}^{(1)}(\theta^*)\leq \widetilde{v}^{(2)}(\theta^*)\}=\{\widetilde{v}^1_{\theta^*}\leq \widetilde{v}^2_{\theta^*}\}$. By Proposition \ref{P6}, the pair of stopping times $(\tau_1^*,\tau_2^*)$ given by
    	\begin{displaymath}
    		\tau_1^*=\theta^* I_B+\theta_1^* I_{B^c}, \ \tau_2^*=\theta^*_2I_B+\theta^*I_{B^c}
    	\end{displaymath}
    	is optimal for $v(S)$. The proof is complete.
    \end{proof}

	\appendix
	\renewcommand\thesection{Appendix A}
	\section{ }
	\renewcommand\thesection{A}
	
	In this section, we will recall some basic notations and  properites of the so-called ``$\mathbb{F}$-expectation" introduced in \cite{BY1}. Roughly speaking, the $\mathbb{F}$-expectation is a nonlinear expectation defined on a subspace of $L^0(\mathcal{F}_T)$, which satisfies the following algebraic properties.
	
	\begin{definition}\label{D2.1}
		Let $\mathscr{D}_T$ denote the collection of all non-empty subsets $\Lambda$ of $L^0(\mathcal{F}_T)$ satisfying:
		\begin{description}
			\item[(D1)] $0,1\in\Lambda$;
			\item[(D2)] for any $\xi,\eta\in\Lambda$ and $A\in\mathcal{F}_T$, both $\xi+\eta$ and $I_A \xi$ belong to $\Lambda$;
			\item[(D3)] for any $\xi,\eta\in L^0(\mathcal{F}_T)$ with $0\leq \xi\leq \eta$, a.s., if $\eta\in\Lambda$, then $\xi\in\Lambda$.
		\end{description}
	\end{definition}
	
	\begin{definition}\label{D2.2}
		An $\mathbb{F}$-consistent nonlinear expectation ($\mathbb{F}$-expectation for short) is a pair $(\mathcal{E},\Lambda)$ in which $\Lambda\in\mathscr{D}_T$ and $\mathcal{E}$ denotes a family of operators $\{\mathcal{E}_t[\cdot]:\Lambda\mapsto\Lambda_t:=\Lambda\cap L^0(\mathcal{F}_t)\}_{t\in[0,T]}$ satisfying the following hypothesis for any $\xi,\eta\in\Lambda$ and $t\in[0,T]$:
		\begin{description}
			\item[(A1)] ``Monotonicity (positively strict)": $\mathcal{E}_t[\xi]\leq \mathcal{E}_t[\eta]$, a.s. if $\xi\leq \eta$, a.s.; Moreover, if $0\leq \xi\leq \eta$ a.s. and $\mathcal{E}_0[\xi]=\mathcal{E}_0[\eta]$, then $\xi=\eta$, a.s.;
			\item[(A2)] ``Time consistency": $\mathcal{E}_s[\mathcal{E}_t[\xi]]=\mathcal{E}_s[\xi]$, a.s. for any $0\leq s\leq t\leq T$;
			\item[(A3)] ``Zero-one law": $\mathcal{E}_t[\xi I_A]=\mathcal{E}_t[\xi]I_A$, a.s. for any $A\in\mathcal{F}_t$;
			\item[(A4)] ``Translation invariance": $\mathcal{E}_t[\xi+\eta]=\mathcal{E}_t[\xi]+\eta$, a.s. if $\eta\in\Lambda_t$.
		\end{description}
	\end{definition}
	
	For notional simplicity, we will substitute $\mathcal{E}[\cdot]$ for $\mathcal{E}_0[\cdot]$. We denote the domain $\Lambda$ by Dom$(\mathcal{E})$ and introduce the following subsets of Dom$(\mathcal{E})$:
	\begin{align*}
	\textrm{Dom}_\tau(\mathcal{E})&:=\textrm{Dom}(\mathcal{E})\cap L^0(\mathcal{F}_\tau), \ \forall \tau\in\mathcal{S}_0,\\
	\textrm{Dom}^+(\mathcal{E})&:=\{\xi\in \textrm{Dom}(\mathcal{E}):\xi\geq 0, \textrm{ a.s.}\},\\
	\textrm{Dom}^{*}(\mathcal{E})&:=\{\xi\in \textrm{Dom}(\mathcal{E}):\xi\geq c, \textrm{ a.s. for some }c=c(\xi)\in\mathbb{R}\}.
	\end{align*}
	
	\begin{definition}\label{D2.3}
		\begin{description}
			\item[(1)] An $\mathbb{F}$-adapted process $X=\{X_t\}_{t\in[0,T]}$ is called an ``$\mathcal{E}$-process" if $X_t\in$Dom$(\mathcal{E})$, for any $t\in[0,T]$;
			\item[(2)] An $\mathcal{E}$-process is said to be an $\mathcal{E}$-supermartingale (resp., $\mathcal{E}$-martingale, $\mathcal{E}$-submartingale) if for any $0\leq s\leq t\leq T$, $\mathcal{E}_s[X_t]\leq (\textrm{resp. } =,\geq) X_s$, a.s.
		\end{description}
	\end{definition}
	
	For any $\mathbb{F}$-adapted process $X$, its right-limit process is defined as follows:
	\begin{displaymath}
	X_t^+:=\liminf_{n\rightarrow\infty}X_{q_n^+(t)}, \textrm{ for any } t\in[0,T],
	\end{displaymath}
	where $q_n^+(t)=\frac{[2^n t]}{2^n}T$. Let $X$ be an $\mathcal{E}$-process. For any stopping time $\tau\in \mathcal{S}^F_0$, where $\mathcal{S}_0^F$ is the collection of all stopping times taking values in a finite set, by Condition (D2) in Definition \ref{D2.1}, it is easy to check that $X_\tau\in$Dom$_\tau(\mathcal{E})$. For any $\xi\in$Dom$(\mathcal{E})$, $\{X^\xi_t\}_{t\in[0,T]}$ is an $\mathcal{E}$-process, where $X_t^\xi=\mathcal{E}_t[\xi]$. Therefore, for any $\tau\in\mathcal{S}_0^F$, we may define an operator $\mathcal{E}_\tau[\cdot]:\textrm{Dom}(\mathcal{E})\mapsto\textrm{Dom}_\tau(\mathcal{E})$ by
	\begin{displaymath}
	\mathcal{E}_\tau[\xi]:=X^\xi_\tau, \textrm{ for any } \xi\in\textrm{Dom}(\mathcal{E}).
	\end{displaymath}
	
	In order to make the operator $\mathcal{E}_\tau[\cdot]$ well-defined for any stopping time $\tau$, we need to put the following hypotheses on the $\mathbb{F}$-expectation and the associated domain Dom$(\mathcal{E})$.
	
	\begin{description}
		\item[(H0)] For any $A\in\mathcal{F}_T$ with $P(A)>0$, we have $\lim_{n\rightarrow\infty}\mathcal{E}[nI_A]=\infty$;
		\item[(H1)] For any $\xi\in\textrm{Dom}^+(\mathcal{E})$ and any $\{A_n\}_{n\in\mathbb{N}}\subset \mathcal{F}_T$ with $\lim_{n\rightarrow\infty}\uparrow I_{A_n}=1$, a.s., we have $\lim_{n\rightarrow\infty}\uparrow\mathcal{E}[\xi I_{A_n}]=\mathcal{E}[\xi]$;
		\item[(H2)] For any $\xi,\eta\in\textrm{Dom}^+(\mathcal{E})$ and any $\{A_n\}_{n\in\mathbb{N}}\subset \mathcal{F}_T$ with $\lim_{n\rightarrow\infty}\downarrow I_{A_n}=0$, a.s., we have $\lim_{n\rightarrow\infty}\downarrow\mathcal{E}[\xi+\eta I_{A_n}]=\mathcal{E}[\xi]$;
		\item[(H3)] For any $\xi\in\textrm{Dom}^+(\mathcal{E})$ and $\tau\in\mathcal{S}_0$, $X^{\xi,+}_\tau\in \textrm{Dom}^+(\mathcal{E})$;
		\item[(H4)] $\textrm{Dom}(\mathcal{E})\in \widetilde{\mathscr{D}_T}:=\{\Lambda\in\mathscr{D}_T:\mathbb{R}\subset\Lambda\}$.
	\end{description}
	
	Under the above assumptions, \cite{BY1} shows that the process $\{X_t^{\xi,+}\}_{t\in[0,T]}$ is an RCLL modification of $\{X_t^\xi\}_{t\in[0,T]}$ for any $\xi\in \textrm{Dom}^+(\mathcal{E})$. Then for any stopping time $\tau\in\mathcal{S}_0$, the conditional $\mathbb{F}$-expectation of $\xi\in \textrm{Dom}^+(\mathcal{E})$ at $\tau$ is given by
	\begin{displaymath}
	\widetilde{\mathcal{E}}_\tau[\xi]:=X^{\xi,+}_\tau.
	\end{displaymath}
	It is easy to check that $\widetilde{\mathcal{E}}_\tau[\cdot]$ is an operator from $\textrm{Dom}^+(\mathcal{E})$ to $\textrm{Dom}^+(\mathcal{E})_\tau:=\textrm{Dom}^+(\mathcal{E})\cap L^0(\mathcal{F}_\tau)$. Furthermore, $\{\widetilde{\mathcal{E}}_t[\cdot]\}_{t\in[0,T]}$ defines an $\mathbb{F}$-expectation and for any $\xi\in \textrm{Dom}^+(\mathcal{E})$, $\{\widetilde{\mathcal{E}}_t[\xi]\}_{t\in[0,T]}$ is an RCLL modification of $\{\mathcal{E}_t[\xi]\}_{t\in[0,T]}$. For simplicity, we still denote $\widetilde{\mathcal{E}}_t[\cdot]$ by $\mathcal{E}_t[\cdot]$ and it satisfies the following properties.
	
	\begin{proposition}\label{P2.7}
		For any $\xi,\eta\in \textrm{Dom}^+(\mathcal{E})$ and $\tau\in \mathcal{S}_0$, it holds that
		\begin{description}
			\item[(1)] ``Monotonicity (positively strict)": $\mathcal{E}_\tau[\xi]\leq \mathcal{E}_\tau[\eta]$, a.s. if $\xi\leq \eta$, a.s.; Moreover, if $\mathcal{E}_\sigma[\xi]=\mathcal{E}_\sigma[\eta]$, a.s. for some $\sigma\in\mathcal{S}_0$, then $\xi=\eta$, a.s.;
			\item[(2)] ``Time consistency": $\mathcal{E}_\sigma[\mathcal{E}_\tau[\xi]]=\mathcal{E}_\sigma[\xi]$, a.s. for any $\tau,\sigma\in\mathcal{S}_0$ with $\sigma\leq \tau$;
			\item[(3)] ``Zero-one law": $\mathcal{E}_\tau[\xi I_A]=\mathcal{E}_\tau[\xi]I_A$, a.s. for any $A\in\mathcal{F}_\tau$;
			\item[(4)] ``Translation invariance": $\mathcal{E}_\tau[\xi+\eta]=\mathcal{E}_\tau[\xi]+\eta$, a.s. if $\eta\in\textrm{Dom}^+_\tau(\mathcal{E})$;
			\item[(5)] ``local property": $\mathcal{E}_\tau[\xi I_A+\eta I_{A^c}]=\mathcal{E}_\tau[\xi]I_A+\mathcal{E}_\tau[\eta]I_{A^c}$, a.s. for any $A\in\mathcal{F}_\tau$;
			\item[(6)] ``Constant-preserving": $\mathcal{E}_\tau[\xi]=\xi$, a.s., if $\xi\in\textrm{Dom}^+_\tau(\mathcal{E})$.
		\end{description}
	\end{proposition}
	
	\begin{proposition}\label{P2.2}
		Let $X$ be a nonnegative $\mathcal{E}$-supermartingale. Then we have
		\begin{description}
			\item[(1)] Assume either that $\esssup_{t\in \mathcal{I}}X_t\in\textrm{Dom}^+(\mathcal{E})$ (where $\mathcal{I}$ is the set of all dyadic rational numbers less than $T$) or that for any sequence $\{\xi_n\}_{n\in\mathbb{N}}\subset\textrm{Dom}^+(\mathcal{E})$ convergences a.s. to some $\xi\in L^0(\mathcal{F}_T)$,
			\begin{displaymath}
			\liminf_{n\rightarrow\infty}\mathcal{E}[\xi_n]<\infty \textrm{ implies } \xi\in \textrm{Dom}^+(\mathcal{E}).
			\end{displaymath}
			Then for any $\tau\in\mathcal{S}_0$, $X^+_\tau\in \textrm{Dom}^+(\mathcal{E})$;
			\item[(2)] If $X_t^+\in \textrm{Dom}^+(\mathcal{E})$ for any $t\in[0,T]$, then $X^+$ is an RCLL $\mathcal{E}$-supermartingale such that for any $t\in[0,T]$, $X_t^+\leq X_t$, a.s.;
			\item[(3)] Moreover, if the function $t\mapsto\mathcal{E}[X_t]$ from $[0,T]$ to $\mathbb{R}$ is right-continuous, then $X^+$ is an RCLL modification of $X$. Conversely, if $X$ has a right-continuous modification, then the function $t\mapsto\mathcal{E}[X_t]$ is right-continuous.
		\end{description}
	\end{proposition}
	
	The Fatou Lemma and the dominated convergence theorem still hold for the  conditional $\mathbb{F}$-expectation $\mathcal{E}_\tau[\cdot]$.
	\begin{proposition}\label{P2.8}
		Let $\{\xi_n\}_{n\in\mathbb{N}}\subset \textrm{Dom}^+(\mathcal{E})$ converge a.s. to some $\xi\in \textrm{Dom}^+(\mathcal{E})$. Then for any $\tau\in\mathcal{S}_0$, we have
		\begin{displaymath}
		\mathcal{E}_\tau[\xi]\leq \liminf_{n\rightarrow\infty}\mathcal{E}_\tau[\xi_n].
		\end{displaymath}
		Furthermore, if there exists an $\eta\in \textrm{Dom}^+(\mathcal{E})$ such that $\xi_n\leq \eta$ a.s. for any $n\in\mathbb{N}$, then the limit $\xi\in \textrm{Dom}^+(\mathcal{E})$ and for any $\tau\in \mathcal{S}_0$, we have
		\begin{displaymath}
		\mathcal{E}_\tau[\xi]= \lim_{n\rightarrow\infty}\mathcal{E}_\tau[\xi_n].
		\end{displaymath}
	\end{proposition}

\end{document}